\providecommand{\orcidlogo}{}

\documentclass[pdflatex,sn-basic,Numbered]{sn-jnl}

\usepackage{graphicx}
\usepackage{amsmath,amssymb,amsfonts}
\usepackage{amsthm}
\usepackage{mathrsfs}
\usepackage{xcolor}
\usepackage{textcomp}
\usepackage{booktabs}
\usepackage{hyperref}

\theoremstyle{thmstyleone}

\newtheorem{theorem}{Theorem}[section]
\newtheorem{proposition}[theorem]{Proposition}
\newtheorem{lemma}[theorem]{Lemma}

\theoremstyle{thmstyletwo}
\newtheorem{remark}{Remark}[section]

\theoremstyle{thmstylethree}

\begin{document}

\title[Critical-Window Fluctuations for the SK Model]{Critical-Window Fluctuations and Disorder Universality for the Sherrington--Kirkpatrick Model}

\author[1]{\fnm{Yu} \sur{Cheng}}\email{12531173@mail.sustech.edu.cn}
\author[2]{\fnm{Song-Hao} \sur{Liu}}\email{liusonghao@dlut.edu.cn}
\author*[3]{\fnm{Qi-Man} \sur{Shao}}\email{shaoqm@sustech.edu.cn}
\author[1]{\fnm{Jing-Yu} \sur{Xu}}\email{12131253@mail.sustech.edu.cn}

\affil[1]{%
  \orgdiv{Department of Statistics and Data Science},
  \orgname{Southern University of Science and Technology},
  \orgaddress{
    \city{Shenzhen},
    \country{China}
  }
}
\affil[2]{%
  \orgdiv{School of Mathematical Sciences},
  \orgname{Dalian University of Technology},
  \orgaddress{
    \city{Dalian},
    \country{China}
  }
}

\affil*[3]{%
  \orgdiv{Department of Statistics and Data Science, Shenzhen International Center for Mathematics},
  \orgname{Southern University of Science and Technology},
  \orgaddress{
    \city{Shenzhen},
    \country{China}
  }
}

\abstract{
\unboldmath
We establish free-energy fluctuation limits for the Ising
Sherrington--Kirkpatrick model in the nonzero parts of its critical window.
For fixed $b\ne0$ and
$\beta_N=1+bN^{-1/3}\sqrt{\log N}$, our main Gaussian orthogonal ensemble (GOE) result is
\[
    \sqrt{\frac6{\log N}}\left(F_{N,\beta_N}-N\,\mathrm{FE}(\beta_N)+\frac{\log N}{12}\right)\xrightarrow{d}G+\sqrt{\frac32}\,b_+TW_1,
\]
where $G$ is standard Gaussian, $TW_1$ has the real Tracy--Widom law,
and $G$ is independent of $TW_1$, and $\mathrm{FE}$ denotes the limit of spherical Sherrington--Kirkpatrick free-energy.
Additionally, we show that in a moderately supercritical regime
\[
    \frac{2}{N^{1/3}(\beta_N-1)}\left(F_{N,\beta_N}-N\,\mathrm{FE}(\beta_N)+\frac{\log N}{12}\right)\xrightarrow{d}TW_1.
\]
We also show that the above results remains valid for independent, not necessarily
identically distributed, disorder matrices whose first three moments match the
Gaussian law and whose fourth moments satisfy an averaged bound.}



\maketitle

\section{Introduction}

Spin glasses are disordered magnetic systems in which random, competing interactions prevent the spins from forming a simple ordered state. 
The Sherrington--Kirkpatrick (SK) model, introduced by Sherrington and Kirkpatrick~\cite{SherringtonKirkpatrick}, is the canonical mean-field model of this phenomenon; see also Binder and Young~\cite{BinderYoung}. 
Its free energy is the central quantity describing the equilibrium behavior of the system: its large-$N$ limit determines the macroscopic thermodynamics, 
while its fluctuations measure the effect of disorder and become especially informative near a phase transition.

We now define the model and its free energy. For each $N$, let 
$W\sim\mathrm{GOE}(N)$.  Thus $W\in\mathbb R^{N\times N}$ is symmetric, its
entries on and above the diagonal are independent, and
\begin{align*}
    W_{ij}\sim\mathcal N\left(0,\frac1N\right),\quad i<j,
    \qquad
    W_{ii}\sim\mathcal N\left(0,\frac2N\right).
\end{align*}

Let $\Sigma_N=\{-1,1\}^N$ be the Ising configuration space.  The Hamiltonian is
\begin{align*}
    H_N(x)=\frac12x^TWx,
\end{align*}
and, at inverse temperature $\beta\geq0$, the normalized partition function
and corresponding free energy are defined as
\begin{align}\label{full GOE convention}
    Z_{N,\beta}=2^{-N}\sum_{x\in\Sigma_N}e^{\beta H_N(x)},
    \qquad
    F_{N,\beta}=\log Z_{N,\beta}.
\end{align}

Two basic questions concern the large-$N$ behavior of the free energy.  The
first is whether the free-energy density $N^{-1}F_{N,\beta}$ converges to a
deterministic limit.  The second is how the random free energy
$F_{N,\beta}$ fluctuates around its leading-order value.  This fluctuation
problem is a second-order refinement of the thermodynamic limit: it identifies
the scale of sample-to-sample disorder fluctuations and, at a phase
transition, can reveal the mechanism responsible for the change of phase.

There is a rich literature on both questions.
Parisi~\cite{Parisi} predicted the limiting free-energy density through a variational
formula. Guerra~\cite{Guerra} proved one side of this formula, and Talagrand~\cite{TalagrandParisi} proved the matching bound for the SK model. 
Panchenko~\cite{PanchenkoParisi} later extended the result to general mixed $p$-spin models. 
The first-order limit is therefore well understood.  
The fluctuation problem is much less complete.
For a recent account of free-energy fluctuations in the SK, spherical SK, and related two-spin models, see the survey \cite{CollinsWoodfinLe}.
For every fixed $0<\beta<1$, order-one Gaussian fluctuations for the zero-diagonal SK model were proved in \cite{AizenmanLebowitzRuelle,CometsNeveu}.  
Guerra and Toninelli obtained a related free-energy fluctuation result in the presence of a general external field, at a different fluctuation scale \cite{GuerraToninelli}.
For the full-GOE model defined above, the high-temperature central limit theorem takes the form
\begin{align}\label{high temperature CLT}
    F_{N,\beta}-\frac{N\beta^2}{4}
    \xrightarrow{d}
    \mathcal N\left(
        \frac14\log(1-\beta^2),
        -\frac12\log(1-\beta^2)
    \right).
\end{align}

The critical inverse temperature is $\beta_c=1$.  For every fixed $\beta>0$, Chatterjee~\cite{chatterjee2009disorder} proved the superconcentration bound
$\operatorname{Var}(F_{N,\beta})\leq C(\beta)N/\log N$.
At criticality, Talagrand's overlap estimate \cite[Chapter~11]{TalagrandMeanFieldII}, combined with Chatterjee's variance representation as explained in \cite{ChenLam}, gives the earlier bound
$\operatorname{Var}(F_{N,1})=O(\sqrt N)$.  Chen and Lam sharpened this to
$\operatorname{Var}(F_{N,1})=O((\log N)^2)$ and obtained a near-critical bound from the low-temperature side at a polynomial rate~\cite{ChenLam}.
A replica calculation predicted the sharper asymptotic
$\operatorname{Var}(F_{N,1})=\frac16\log N+O(1)$ \cite{Aspelmeier}, which was established by Du and Huang~\cite{DuHuang}.
For fixed $\beta>1$, however, the correct fluctuation scale remains a significant open problem; competing physical predictions are discussed in \cite{AspelmeierBilloireMarinariMoore,ParisiRizzo}.
This makes the neighborhood of $\beta=1$ a natural place to study how the
high-temperature Gaussian law transfers to the real Tracy--Widom law.


\subsection{The SSK model and the critical window}

A tractable analogue is the spherical SK model (SSK), introduced by
Kosterlitz, Thouless, and Jones~\cite{KosterlitzThoulessJones}.  The SSK model retains the mean-field
two-spin interaction while its rotational symmetry makes its free energy
amenable to random-matrix analysis.
Its leading-order free energy was rigorously computed by Talagrand~\cite{TalagrandSpherical}.
Its state space is 
\begin{align*}
    S_N=\{x\in\mathbb R^N:\|x\|_2^2=N\}.
\end{align*}

Let $\nu_N$ be the uniform probability measure on $S_N$.  Using the same
Hamiltonian, define the partition function and corresponding free energy of
the SSK model as
\begin{align*}
    Z_{N,\beta}^{sph}
    =\int_{S_N}e^{\beta H_N(x)}\,d\nu_N(x),
    \qquad
    F_{N,\beta}^{sph}=\log Z_{N,\beta}^{sph}.
\end{align*}

Baik and Lee~\cite{BaikLee} proved a sharp change in the SSK free-energy fluctuations.  
After proper scaling, $F_{N,\beta}^{sph}$ converges to Gaussian and Tracy--Widom laws in the high- and low-temperature regimes, respectively.
By comparing the orders of the variances of $F_{N,\beta}^{sph}$ in these two regimes, Baik and Lee~\cite{BaikLee} predicted the critical window in which the Gaussian law crosses over to the Tracy--Widom law:
\begin{align}\label{critical window}
    \beta_N=1+bN^{-1/3}\sqrt{\log N},
    \qquad b\in\mathbb R.
\end{align}

Landon~\cite{Landon} gave the first rigorous analysis of this window.  He
proved the Gaussian limit for $b\leq0$, obtained tightness for each fixed
$b>0$, and recovered the Tracy--Widom limit when $b\to+\infty$ while $\beta$ is bounded. 

Johnstone, Klochkov, Onatski, and Pavlyshyn~\cite{JKOP} then identified the limit for every fixed
$b\in\mathbb R$ for the real Wigner inputs satisfying the moment-matching
and regularity hypotheses of their theorem.

Define the limit of the SSK free energy by
\begin{align*}
    \mathrm{FE}(\beta)=
    \begin{cases}
        \beta^2/4, & 0<\beta\leq1,\\[2mm]
        \beta-\dfrac12\log\beta-\dfrac34, & \beta>1.
    \end{cases}
\end{align*}
Set $b_+=\max\{b,0\}$.  Let $G\sim\mathcal N(0,1)$, and let $TW_1$ be independent of $G$ with the real Tracy--Widom law~\cite{TracyWidom}.
Then
\begin{align}
    \label{sph limit}
    \sqrt{\frac6{\log N}}
    \left(
        F_{N,\beta_N}^{sph}
        -N\,\mathrm{FE}(\beta_N)
        +\frac{\log N}{12}
    \right)
    \xrightarrow{d}
    G+\sqrt{\frac32}\,b_+TW_1.
\end{align}
Thus the limit is Gaussian for $b\leq0$.  For $b>0$, it is the sum of
independent Gaussian and Tracy--Widom terms.

\subsection{Ising SK model and critical window}

For the zero-diagonal Ising SK model, Dey and Kang~\cite{DeyKang} proved a Gaussian central
limit theorem for inverse temperatures approaching $1$ from below.  Their result applies when
$N^{1/3}(1-\beta_N^2)$ converges to a positive number or to infinity.  In
particular, it covers $b<0$ in the window \eqref{critical window}. At the critical point
itself, Du and Huang~\cite{DuHuang} worked with the full-GOE convention \eqref{full GOE convention} and proved
\begin{align*}
    \operatorname{Var}(F_{N,1})=\frac16\log N+O(1)
\end{align*}
and an explicitly centered Gaussian central limit theorem. 

Their critical reweighting is in the spirit of the small subgraph conditioning
method (\cite{RobinsonWormald1992,RobinsonWormald1994}); 
see also its use in mean-field disordered models (for example \cite{AbbeLiSly,BencsHuangLeeLiuRegts}).

Their main idea is a critical reweighting that compares the Ising and SSK partition functions through a common overlap kernel.

They defineed $X_{N,\beta}=Z_{N,\beta}/Z_{N,\beta}^{sph}$ and proved
\begin{align*}
    \mathbb E\left[(X_{N,1}-1)^2\right]\leq CN^{-1/3}.
\end{align*}
Hence
\begin{align*}
    F_{N,1}-F_{N,1}^{sph}=\log X_{N,1}=o_{\mathbb P}(1),
\end{align*}
and the SSK critical central limit theorem transfers to the Ising SK model by Slutsky's lemma.

The same localization argument does not directly extend to the low-temperature side of the window. 
They expect that more delicate random-matrix input will be helpful to overcome this problem.

Recent complementary work has refined the understanding of critical SK fluctuations and overlaps. 
Logarithmic upper and iterated-logarithmic lower bounds for the critical variance were obtained in \cite{Schertzer2026}, 
an independent proof of the critical variance asymptotics in a different normalization appears in \cite{Chen2026}, 
and the critical overlap law was subsequently identified in \cite{DuHuangOverlap}.

\subsection{Our main contributions}

Our first result resolves the critical-window mixed-limit problem raised in \cite{DuHuang}.  
For every fixed $b\ne0$, and in the moderately supercritical regime stated below, the Ising free energy has the same limit as its SSK counterpart.  
The case $b<0$ follows from strong log-concavity, while the new issue is $b>0$, where the original log-concavity estimate no longer directly provides the required localization.
We prove a cubic localization bound by reducing the relevant free-energy increment to four random terms, each controlled by a separate random-matrix estimate.

Our second contribution is a quantitative transfer of these fluctuation theorems from GOE input to general symmetric random matrices.
For each $N$, let
$(\xi_{ij})_{1\leq i\leq j\leq N}$ be independent, but not necessarily
identically distributed, real random variables.  Set
\begin{align*}
    W_{ij}^{\boldsymbol\xi}=W_{ji}^{\boldsymbol\xi}=\frac{\xi_{ij}}{\sqrt N},\quad i<j,
    \qquad
    W_{ii}^{\boldsymbol\xi}=\frac{\sqrt2\,\xi_{ii}}{\sqrt N}.
\end{align*}

Assume that the first three moments of every $\xi_{ij}$ match those of a standard normal variable and that the fourth moments satisfy the uniform averaged bound
\begin{align*}
    \sup_N\frac1{N^2}\sum_{1\leq i\leq j\leq N}\mathbb E|\xi_{ij}|^4<\infty.
\end{align*}

At the level of the limiting free-energy density, Carmona and Hu~\cite{CarmonaHu} established disorder universality under a uniform third-moment condition; 
their Remark~3 also allows independent, non-identically distributed couplings.  
At the fluctuation level, Choo, Han, and Lee~\cite{ChooHanLee} recently established high-temperature Gaussian limits for generalized SK models under sub-Gaussian assumptions.  
Under the assumptions above, we prove critical-window fluctuation transfer for independent, not necessarily identically distributed couplings.  
The proof is a quantitative Lindeberg replacement argument, which also yields the fixed-critical-temperature and fixed-high-temperature theorems.

Furthermore, our method covers the moderately supercritical regime
\begin{align*}
    \beta_N=1+b_NN^{-1/3}\sqrt{\log N},\qquad
    b_N\to\infty,\qquad
    b_N=o\left(\frac{N^{1/12}}{\sqrt{\log N}}\right).
\end{align*}
In this moderately supercritical regime, the same comparison transfers the SSK Tracy--Widom limit, with the centering and normalization in Theorem~\ref{GOE mixed limit}\textup{(ii)}, 
to both the Ising and generalized models.

The remainder of the paper is organized as follows.  Section~2 states the fluctuation theorems.  
Section~3 reduces the main results to the sphere--cube and four-moment comparison propositions.
Sections~4 and~5 establish the variational, soft-edge, and density estimates that yield the positive-side cubic localization.  
Section~6 proves the random-matrix counting and edge inputs, and Section~7 treats the case $b<0$.

\section{Main results}

\subsection{Critical-window mixed limits}
Our first main result is the mixed-limit theorem for the Ising SK model, which
is parallel to the SSK case proved by Landon~\cite{Landon} and by Johnstone, Klochkov, Onatski, and Pavlyshyn~\cite{JKOP}, 
and answers the open question in Remark~2.3 of \cite{DuHuang}.
Our method is to prove $L^2$ convergence of $X_{N,\beta}$ at critical window.

\begin{theorem}[GOE critical and moderately supercritical limits]\label{GOE mixed limit}
    \textup{(i)} Let
    \begin{align*}
        \beta_N=1+bN^{-1/3}\sqrt{\log N}, \qquad b\in\mathbb{R}\setminus\{0\}.
    \end{align*}

    Then
    \begin{align*}
        \sqrt{\frac{6}{\log N}}\left(F_{N,\beta_N}-N\,\mathrm{FE}(\beta_N)+\frac{\log N}{12}\right) \xrightarrow{d} G+\sqrt{\frac32}\,b_+TW_1,
    \end{align*}
    where $G\sim\mathcal{N}(0,1)$, $TW_1$ has the GOE Tracy--Widom distribution, and $G$ and $TW_1$ are independent. $\mathrm{FE}(\beta_N)$ is the limit of SSK free energy.

    {
    \textup{(ii)} Let
    \begin{align*}
        \beta_N=1+b_NN^{-1/3}\sqrt{\log N},
    \end{align*}
    where
    \begin{align}\label{moderately supercritical condition}
        b_N\longrightarrow\infty,
        \qquad
        b_N=o\left(\frac{N^{1/12}}{\sqrt{\log N}}\right).
    \end{align}

    Then
    \begin{align*}
        \frac{2}{N^{1/3}(\beta_N-1)}
        \left(F_{N,\beta_N}-N\,\mathrm{FE}(\beta_N)+\frac{\log N}{12}\right)
        \xrightarrow{d}TW_1.
    \end{align*}
    }
\end{theorem}

\begin{remark}[Normalization in the moderately supercritical regime]\label{rem:Landon-normalization}
Denote $\lambda_1$ be the largest eigenvalue of $W$.
Under the convention $N^{2/3}(\lambda_1-2)\xrightarrow{d}TW_1$, recombining the saddle expansion (5.9) and contour estimate (5.10) in Landon~\cite{Landon}, including the prefactor $N/2$ in the contour representation, gives the leading edge contribution $N(\beta_N-1)(\lambda_1-2)/2$.  This suggests the normalization $2/[N^{1/3}(\beta_N-1)]$, which is also consistent with \cite{BaikLee,JKOP}.  The displayed coefficients in (5.8), (5.11), and (5.17) of \cite{Landon} do not appear to be mutually consistent under this convention.
\end{remark}

For $b=0$, the corresponding GOE critical result is due to Du and Huang~\cite{DuHuang}.  The same work also observes that its localization argument can be adapted to the case $b<0$.

\subsection{General disorder: Universality}

For every $N$, let
\begin{align*}
    \mathcal E_N=\{\mathbf{i}=(r,s):1\leq r\leq s\leq N\}, \qquad M_N=|\mathcal E_N|=\frac{N(N+1)}{2}.
\end{align*}

Equip $\mathcal E_N$ with the lexicographic order $\prec$.  Thus, if $\mathbf{i}=(r,s)$ and $\mathbf{j}=(r',s')$ belong to $\mathcal E_N$, then
\begin{align*}
    \mathbf{i}\prec\mathbf{j} \quad\Longleftrightarrow\quad r<r' \quad\text{or}\quad (r=r'\ \text{and}\ s<s').
\end{align*}

For $\mathbf{i}=(r,s)\in\mathcal E_N$, set
\begin{align*}
    a_{\mathbf{i}}=\begin{cases}
        2, & r<s,\\
        \sqrt2, & r=s.
    \end{cases}
\end{align*}

Let $(\xi_{\mathbf{i}})_{\mathbf{i}\in\mathcal E_N}$ be independent real random variables, regarded as standardized matrix coordinates.  Define the symmetric disorder matrix $W^{\boldsymbol\xi}$ by
\begin{align*}
    W^{\boldsymbol\xi}_{rs}=W^{\boldsymbol\xi}_{sr}=\frac{\xi_{(r,s)}}{\sqrt N},\qquad r<s,
    \qquad
    W^{\boldsymbol\xi}_{rr}=\frac{\sqrt2\,\xi_{(r,r)}}{\sqrt N},\qquad 1\leq r\leq N.
\end{align*}

For a real random variable $Y$, set
\begin{align*}
    \tau_N(Y)=\mathbb{E}\left[|Y|^4\min\left\{1,\frac{|Y|}{\sqrt N}\right\}\right],
\end{align*}

Let
\begin{align*}
    \Xi_N=\frac1{N^2}\sum_{\mathbf{i}\in\mathcal E_N}\left(\left|\mathbb{E}[\xi_{\mathbf{i}}^4]-3\right|+\tau_N(\xi_{\mathbf{i}})\right).
\end{align*}

Define the Hamiltonian and the corresponding free energy as
\begin{align*}
    H_N^{\boldsymbol\xi}(\sigma)
    =\frac12\sigma^T W^{\boldsymbol\xi}\sigma
    =\frac{1}{2\sqrt{N}}\sum_{\mathbf{i}=(r,s)\in\mathcal E_N}a_{\mathbf{i}}\xi_{\mathbf{i}}\sigma_r\sigma_s,
    \qquad
    F_{N,\beta}^{\boldsymbol\xi}=\log\left(2^{-N}\sum_{\sigma\in\Sigma_N}e^{\beta H_N^{\boldsymbol\xi}(\sigma)}\right).
\end{align*}

We impose the first three moment matching condition for every standardized coordinate that
\begin{align}\label{first three moment matching}
    \mathbb{E}\xi_{\mathbf{i}}=0,\qquad\mathbb{E}[\xi_{\mathbf{i}}^2]=1,\qquad\mathbb{E}[\xi_{\mathbf{i}}^3]=0,
    \qquad \mathbf{i}\in\mathcal E_N.
\end{align}

\begin{theorem}[General critical and moderately supercritical limits]\label{four moment mixed limit}
    Assume \eqref{first three moment matching} and
    \begin{align}\label{critical moment conditions}
        \frac1{N^2}\sum_{\mathbf{i}\in\mathcal E_N}\mathbb E|\xi_{\mathbf{i}}|^4\leq C.
    \end{align}

    \textup{(i)} For every fixed $b\in\mathbb{R}\setminus\{0\}$ and
    \begin{align*}
        \beta_N=1+bN^{-1/3}\sqrt{\log N},
    \end{align*}
    one has
    \begin{align*}
        \sqrt{\frac{6}{\log N}}\left(F_{N,\beta_N}^{\boldsymbol\xi}-N\,\mathrm{FE}(\beta_N)+\frac{\log N}{12}\right) \xrightarrow{d} G+\sqrt{\frac32}\,b_+TW_1.
    \end{align*}

    Here $G\sim\mathcal N(0,1)$, $TW_1$ has the real Tracy--Widom law, and $G$ and $TW_1$ are independent. $\mathrm{FE}(\beta_N)$ is the limit of SSK free energy.

    \textup{(ii)} Let
        \begin{align*}
            \beta_N=1+b_NN^{-1/3}\sqrt{\log N},
        \end{align*}
    where $b_N$ satisfies \eqref{moderately supercritical condition}.

    Then
    \begin{align*}
        \frac{2}{N^{1/3}(\beta_N-1)}
        \left(F_{N,\beta_N}^{\boldsymbol\xi}-N\,\mathrm{FE}(\beta_N)+\frac{\log N}{12}\right)
        \xrightarrow{d}TW_1.
    \end{align*}
\end{theorem}

\begin{theorem}[Fixed high-temperature fluctuation]\label{fixed high temperature}
    Assume \eqref{first three moment matching} and
    \begin{align}\label{comparison condition} 
        \Xi_N\longrightarrow0,
    \end{align}

    Then, for every fixed $0<\beta<1$,
    \begin{align*}
        F_{N,\beta}^{\boldsymbol\xi}-\frac{N\beta^2}{4}\xrightarrow{d}\mathcal{N}\left(m_\beta,\widehat v_\beta\right),
    \end{align*}
    where
    \begin{align*}
        m_\beta=\frac14\log(1-\beta^2),\qquad \widehat v_\beta=-\frac12\log(1-\beta^2).
    \end{align*}
\end{theorem}

\begin{theorem}[Fixed critical temperature]\label{fixed critical temperature}
    Assume \eqref{first three moment matching} and \eqref{critical moment conditions}.
    Then
    \begin{align*}
        \frac{F_{N,1}^{\boldsymbol\xi}-N/4+(\log N)/12} {\sqrt{(\log N)/6}} \xrightarrow{d}\mathcal{N}(0,1).
    \end{align*}
\end{theorem}

\section{Proofs of the main results}

The proofs of the main theorems are reduced to a comparison between the Ising partition function and its spherical counterpart. 
Following the comparison viewpoint in \cite{DuHuang}, we isolate this step as an $L^2$ estimate for $X_{N,\beta}=Z_{N,\beta}/Z_{N,\beta}^{sph}$; 
once this ratio converges to one, the mixed-limit laws for the GOE model transfer directly from the spherical results. 
The proposition below records the form needed throughout the critical window.

\begin{proposition}[Sphere--cube comparison]\label{L2 comparison}
    Let
    \begin{align*}
        X_{N,\beta}=Z_{N,\beta}/Z_{N,\beta}^{sph}.
    \end{align*}

    Suppose that
    \begin{align*}
        \beta_N=1+b_NN^{-1/3}\sqrt{\log N},
    \end{align*}
    where, for some $b_0>0$,
    \begin{align*}
        b_N\geq b_0,
        \qquad
        b_N=o\left(\frac{N^{1/12}}{\sqrt{\log N}}\right).
    \end{align*}
    Then, for all sufficiently large $N$,
    \begin{equation}
        \begin{aligned}\label{moderately supercritical L2 bound}
            \mathbb{E}\left[(X_{N,\beta_N}-1)^2\right]
            &\leq C\left(N^{-1}+(\beta_N-1)+(\beta_N-1)^2+N(\beta_N-1)^4\right)\\
            &\quad+N^C\exp\{-cN(\beta_N-1)^3\}+N^C\exp\{-cN(\beta_N-1)^2\}.
        \end{aligned}
    \end{equation}

    If $b_N=b<0$, then, for all sufficiently large $N$,
    \begin{align*}
        \mathbb{E}\left[(X_{N,\beta_N}-1)^2\right]\leq C_bN^{-1/3}(\log N)^{-1/4}.
    \end{align*}
    Consequently, $X_{N,\beta_N}\xrightarrow{L^2}1$ in both preceding cases.
\end{proposition}

The case $b>0$ is proved in Subsection~\ref{subsec:positive-proposition-proofs}, 
with the auxiliary estimates in Sections~\ref{sec:further-positive-lemmas}--\ref{sec:random-matrix-input-proofs}; 
the case $b=0$ of the corresponding GOE comparison is due to Du and Huang~\cite{DuHuang}; and the case $b<0$ is proved in Section~\ref{sec:negative-window-proof}.

\begin{proof}[Proof of Theorem~\ref{GOE mixed limit}\textup{(i)}]
    Proposition~\ref{L2 comparison} gives $F_{N,\beta_N}-F_{N,\beta_N}^{sph}=\log X_{N,\beta_N}=o_{\mathbb{P}}(1)$.
    The result follows from \eqref{sph limit} and Slutsky's lemma.
\end{proof}

\begin{proof}[Proof of Theorem~\ref{GOE mixed limit}\textup{(ii)}]
    {
    Proposition~\ref{L2 comparison} gives
    \begin{align*}
        X_{N,\beta_N}\xrightarrow{L^2}1.
    \end{align*}

    Hence
    \begin{align*}
        F_{N,\beta_N}-F_{N,\beta_N}^{sph}=\log X_{N,\beta_N}=o_{\mathbb P}(1).
    \end{align*}

    The condition \eqref{moderately supercritical condition} implies $\beta_N\to1$ and lies in the moderately supercritical range considered by Landon~\cite{Landon}.  In the present total-free-energy normalization, and with the normalization discussed in Remark~\ref{rem:Landon-normalization}, the corresponding spherical result gives
    \begin{align*}
        \frac{2}{N^{1/3}(\beta_N-1)}
        \left(F_{N,\beta_N}^{sph}-N\,\mathrm{FE}(\beta_N)+\frac{\log N}{12}\right)
        \xrightarrow{d}TW_1.
    \end{align*}

    Since
    \begin{align*}
        N^{1/3}(\beta_N-1)=b_N\sqrt{\log N}\longrightarrow\infty,
    \end{align*}
    the result follows from Slutsky's lemma.
    }
\end{proof}

\begin{proof}[Proof of Theorem~\ref{four moment mixed limit}\textup{(i)}]
    For random variables $U,V$, write
    \begin{align*}
        d_{\mathrm K}(U,V)=\sup_{x\in\mathbb R}\left|\mathbb P(U\leq x)-\mathbb P(V\leq x)\right|.
    \end{align*}

    Let
    \begin{align*}
        s_N=\sqrt{\frac{\log N}{6}},\qquad a_N=N\,\mathrm{FE}(\beta_N)-\frac{\log N}{12},\qquad Z_b=G+\sqrt{\frac32}\,b_+TW_1.
    \end{align*}

    Theorem~\ref{GOE mixed limit} gives
    \begin{align*}
        \frac{F_{N,\beta_N}-a_N}{s_N}\xrightarrow{d}Z_b.
    \end{align*}

    Since $Z_b$ is the convolution of a standard Gaussian random variable with an independent random variable, 
    it admits a Lebesgue density bounded by $(2\pi)^{-1/2}$. 
    Thus
    \begin{align}\label{convergence of gaussian case} 
        d_{\mathrm K}\left( \frac{F_{N,\beta_N}-a_N}{s_N},Z_b \right)\longrightarrow0.
    \end{align}

    We record the quantitative replacement estimate here,
    where it is first used.
    
    \begin{proposition}[Quantitative four-moment transfer]\label{four moment comparison}
        Assume \eqref{first three moment matching}. Fix $B<\infty$. If
        $|\beta_N|\leq B$, $s_N\geq1$, and $a_N$ is deterministic, then every
        $\varphi\in C_b^5(\mathbb{R})$ with bounded first five derivatives
        satisfies
        \begin{align}\label{four moment comparison estimate}
            \left|\mathbb{E}\varphi\left(\frac{F_{N,\beta_N}^{\boldsymbol\xi}-a_N}{s_N}\right)-\mathbb{E}\varphi\left(\frac{F_{N,\beta_N}-a_N}{s_N}\right)\right|
            \leq C_B\left(\Xi_N+N^{-1/2}\right)
            \sum_{r=1}^{5}s_N^{-r}\|\varphi^{(r)}\|_\infty.
        \end{align}

        Moreover, for every continuous random variable $Z$ with density bounded
        by $M$,
        \begin{align}\label{four moment Kolmogorov comparison}
            d_{\mathrm K}\left(\frac{F_{N,\beta_N}^{\boldsymbol\xi}-a_N}{s_N},Z\right)
            \leq d_{\mathrm K}\left(\frac{F_{N,\beta_N}-a_N}{s_N},Z\right)+C_{B,M}\left[\left(\frac{\Xi_N+N^{-1/2}}{s_N^5}\right)^{1/6}+\left(\frac{\Xi_N+N^{-1/2}}{s_N}\right)^{1/2}
            \right].
        \end{align}
    \end{proposition}

    The proof of Proposition~\ref{four moment comparison} is given in Subsection~\ref{subsec:positive-proposition-proofs}; the Kolmogorov estimate \eqref{four moment Kolmogorov comparison} is the standard smoothing consequence of the smooth test-function estimate \eqref{four moment comparison estimate}.

    Condition \eqref{critical moment conditions} gives
    \begin{align*}
        \Xi_N\leq\frac1{N^2}\sum_{\mathbf{i}\in\mathcal E_N}\left(2\mathbb E|\xi_{\mathbf{i}}|^4+3\right)\leq C.
    \end{align*}

    Hence \eqref{convergence of gaussian case} and \eqref{four moment Kolmogorov comparison} give
    \begin{align*}
        \frac{F_{N,\beta_N}^{\boldsymbol\xi}-a_N}{s_N}\xrightarrow{d}Z_b.
    \end{align*}
\end{proof}

\begin{proof}[Proof of Theorem~\ref{four moment mixed limit}\textup{(ii)}]
    Let
    \begin{align*}
        s_N=\frac12N^{1/3}(\beta_N-1),
        \qquad
        a_N=N\,\mathrm{FE}(\beta_N)-\frac{\log N}{12}.
    \end{align*}

    Theorem~\ref{GOE mixed limit}\textup{(ii)} gives
    \begin{align}\label{TW convergence for GOE}
        \frac{F_{N,\beta_N}-a_N}{s_N}\xrightarrow{d}TW_1.
    \end{align}

    Condition \eqref{critical moment conditions} gives $\Xi_N\leq C$, while $s_N\to\infty$.  
    Since $TW_1$ has a bounded density, \eqref{four moment Kolmogorov comparison} and the preceding convergence give
    \begin{align*}
        d_{\mathrm K}\left(\frac{F_{N,\beta_N}^{\boldsymbol\xi}-a_N}{s_N},TW_1\right)
        \leq d_{\mathrm K}\left(\frac{F_{N,\beta_N}-a_N}{s_N},TW_1\right)+C\left[\left(\frac{\Xi_N+N^{-1/2}}{s_N^5}\right)^{1/6}+\left(\frac{\Xi_N+N^{-1/2}}{s_N}\right)^{1/2}\right].
    \end{align*}

    Combining with \eqref{TW convergence for GOE}, this proves the claimed convergence.
\end{proof}

\begin{proof}[Proof of Theorem~\ref{fixed high temperature}]
    For the fixed high-temperature phase $0<\beta<1$,
    the results of Aizenman, Lebowitz, and Ruelle~\cite{AizenmanLebowitzRuelle} and Comets and Neveu~\cite{CometsNeveu} give
    \begin{align*}
        F_{N,\beta}-\frac{\beta}{2}\sum_{i=1}^{N}W_{ii}-\frac{N\beta^2}{4}
        \xrightarrow{d}
        \mathcal N\left(m_\beta,-\frac12\left(\log(1-\beta^2)+\beta^2\right)\right).
    \end{align*}
    
    As $\frac{\beta}{2}\sum_{i=1}^{N}W_{ii}\sim\mathcal N\left(0,\frac{\beta^2}{2}\right)$,
    we have
    \begin{align}\label{Gaussian convergence for fixed high temperature}
        F_{N,\beta}-\frac{N\beta^2}{4}\xrightarrow{d}\mathcal{N}(m_\beta,\widehat v_\beta).
    \end{align}

    See also the generalized result of Choo, Han, and Lee~\cite{ChooHanLee}.

    The Gaussian limit has a bounded density. 
    Proposition~\ref{four moment comparison} with $s_N=1$ and $a_N=N\beta^2/4$ gives
    \begin{align}\label{four moment for fixed high temperature}
        d_{\mathrm K}\left(F_{N,\beta}^{\boldsymbol\xi}-\frac{N\beta^2}{4},\mathcal N(m_\beta,\widehat v_\beta)\right)
        &\leq d_{\mathrm K}\left(F_{N,\beta}-\frac{N\beta^2}{4},\mathcal N(m_\beta,\widehat v_\beta)\right)\notag\\
        &\quad+C_\beta\left((\Xi_N+N^{-1/2})^{1/6}+(\Xi_N+N^{-1/2})^{1/2}\right).
    \end{align}

    Thus \eqref{comparison condition}, \eqref{high temperature CLT}, \eqref{Gaussian convergence for fixed high temperature}, and
    \eqref{four moment for fixed high temperature} give
    \begin{align*}
        F_{N,\beta}^{\boldsymbol\xi}-\frac{N\beta^2}{4}\xrightarrow{d}\mathcal{N}(m_\beta,\widehat v_\beta).
    \end{align*}
\end{proof}

\begin{proof}[Proof of Theorem~\ref{fixed critical temperature}]
    Put
    \begin{align*}
        s_N=\sqrt{\frac{\log N}{6}},\qquad a_N=\frac{N}{4}-\frac{\log N}{12}.
    \end{align*}

    Du and Huang~\cite{DuHuang} proved that
    \begin{align}\label{critical Gaussian comparison} 
        \frac{F_{N,1}-a_N}{s_N} \xrightarrow{d}\mathcal{N}(0,1).
    \end{align}

    Condition \eqref{critical moment conditions} gives $\Xi_N\leq C$.
    Combining \eqref{four moment Kolmogorov comparison} and \eqref{critical Gaussian comparison} gives
    \begin{align*}
        \frac{F_{N,1}^{\boldsymbol\xi}-a_N}{s_N}\xrightarrow{d}\mathcal N(0,1).
    \end{align*}
\end{proof}

\subsection{\texorpdfstring{Proofs of Proposition~\ref{L2 comparison} for $b>0$ and Proposition~\ref{four moment comparison}}{Proofs of comparison propositions}}\label{subsec:positive-proposition-proofs}

For $x,y\in S_N$, let $q=R(x,y)=x^Ty/N$ be the overlap. 
Throughout the rest of the paper, $\rho_N(q)$ denotes the overlap density of two independent uniform spherical points:
\begin{align*}
    \rho_N(q)=\frac{\Gamma(N/2)}{\sqrt{\pi}\,\Gamma((N-1)/2)}(1-q^2)^{(N-3)/2}{\bf 1}_{\{|q|<1\}}.
\end{align*}

\begin{proof}[Proof of Proposition~\ref{L2 comparison} for $b>0$]
    
    If $\widetilde x,\widetilde y\in S_N$ satisfy $\widetilde x^T\widetilde y=x^Ty$,
    then there is an orthogonal matrix $O$ such that
    \begin{align*}
        \widetilde x=Ox,\qquad \widetilde y=Oy.
    \end{align*}

    For a symmetric matrix $M$, write
    \begin{align*}
        H_M(z)=\frac12z^TMz,\qquad Z_{N,\beta}^{sph}(M)=\int_{S_N}e^{\beta H_M(z)}d\nu_N(z).
    \end{align*}

    Since $Z_{N,\beta}^{sph}(O^TWO)=Z_{N,\beta}^{sph}(W)$ and $O^TWO\overset{d}=W$, we have
    \begin{align*}
        \mathbb{E}\left[\frac{e^{\beta H_W(\widetilde x)+\beta H_W(\widetilde y)}}{Z_{N,\beta}^{sph}(W)^2}\right]
        =\mathbb{E}\left[\frac{e^{\beta H_{O^TWO}(x)+\beta H_{O^TWO}(y)}}{(Z_{N,\beta}^{sph}(O^TWO))^2}\right] 
        =\mathbb{E}\left[\frac{e^{\beta H_W(x)+\beta H_W(y)}}{Z_{N,\beta}^{sph}(W)^2}\right].
    \end{align*}

    Thus this expectation depends on $x,y$ only through $q$. 
    We may therefore define
    \begin{align*}
        J_{\beta}(q)=\mathbb{E}\left[\frac{e^{\beta H_W(x)+\beta H_W(y)}}{Z_{N,\beta}^{sph}(W)^2}\right].
    \end{align*}

    We now prove the proposition using several auxiliary estimates. 
    {
    Throughout this proof, let
    \begin{align*}
        \beta=\beta_N=1+b_NN^{-1/3}\sqrt{\log N},
    \end{align*}
    where $b_N\geq b_0>0$ and $b_N=o(N^{1/12}/\sqrt{\log N})$.
    }
    Fix constants $L>1$ and $M>1$, which will be chosen later, and put
    \begin{align*}
        u_{\beta}=L(1-\beta^{-1}),\quad v_{\beta}=M\sqrt{\beta-1}.
    \end{align*}

    We divide $\left|q\right|\leq 1$ into three parts. For $u_{\beta}\leq \left|q\right|\leq v_{\beta}$, write
    \begin{align*}
        \delta=\beta\left|q\right|-\beta+1.
    \end{align*}

    The main estimate we need is the following exponential upper bound.

    \begin{lemma}\label{tail bound implies L2}
        Let $X_{N,\beta}=Z_{N,\beta}/Z_{N,\beta}^{sph}$.
        Assume that there exist constants $c,C>0$ such that, uniformly for $u_{\beta}\leq \left|q\right|\leq v_{\beta}$,
        \begin{align}\label{main tail bound} 
            \rho_N(q)J_{\beta}(q)\leq N^C\exp\{-cN\delta^3\}.
        \end{align}

        Then
        {
        \begin{align*}
            \mathbb{E}[(X_{N,\beta}-1)^2]
            &\leq C\left(N^{-1}+(\beta-1)+(\beta-1)^2+N(\beta-1)^4\right)\\
            &\quad+N^C\exp\{-cN(\beta-1)^3\}
            +N^C\exp\{-cN(\beta-1)^2\}.
        \end{align*}
        }
    \end{lemma}

    Therefore it remains to prove \eqref{main tail bound}, for which we need an upper bound of $J_{\beta}(q)$ given by the following result.

    \begin{lemma}\label{direct reduction to Delta F}
        For $-1< q< 1$, define
        \begin{align*}
            \Delta F_N^{sph}(q)=F_{N,\beta(1+q)}^{sph}+F_{N,\beta(1-q)}^{sph}-2F_{N,\beta}^{sph}.
        \end{align*}

        There is a numerical constant $C_0>0$ such that, for all $u_{\beta}\leq \left|q\right|\leq v_{\beta}$,
        \begin{align}\label{J direct Delta F bound} 
            J_{\beta}(q)\leq C\sqrt{N}\,\mathbb{E}\exp\{\Delta F_N^{sph}(q)+C_0\Vert W\Vert_{op}\}.
        \end{align}
    \end{lemma}

    Thus the proof of \eqref{main tail bound} turns into estimation of
    \begin{align}
        \label{target exp} \rho_N(q)\mathbb{E}\exp\{\Delta F_N^{sph}(q)+C_0\Vert W\Vert_{op}\}.
    \end{align}

    It remains to prove the following explicit upper bound for \eqref{target exp}.

    \begin{lemma}\label{Delta F exponential upper}
        For every fixed $C_0>0$, there exist constants $c,C>0$, independent of $N$ and $b_N$, such that, uniformly for $u_{\beta}\leq \left|q\right|\leq v_{\beta}$,
        \begin{align}\label{target Delta F op bound} 
            \rho_N(q)\mathbb{E}\exp\{\Delta F_N^{sph}(q)+C_0\Vert W\Vert_{op}\}\leq N^C\exp\{-cN\delta^3\}.
        \end{align}
    \end{lemma}

    The three lemmas above are proved in Section~\ref{sec:further-positive-lemmas}; the proof of Lemma~\ref{Delta F exponential upper} uses auxiliary estimates proved in Sections~\ref{sec:remaining-auxiliary-lemmas}--\ref{sec:random-matrix-input-proofs}.

    By Lemma~\ref{direct reduction to Delta F} and Lemma~\ref{Delta F exponential upper}, the estimate \eqref{main tail bound} holds. 
    {
    Hence Lemma~\ref{tail bound implies L2} gives \eqref{moderately supercritical L2 bound}.
    }
\end{proof}

\begin{proof}[Proof of Proposition~\ref{four moment comparison}]
    Let $W\sim\mathrm{GOE}(N)$ be independent of $(\xi_{\mathbf{i}})_{\mathbf{i}\in\mathcal E_N}$.  For $\mathbf{i}=(r,s)\in\mathcal E_N$, put
    \begin{align*}
        \widehat\xi_{\mathbf{i}}=\begin{cases}
            \sqrt N W_{rs}, & r<s,\\
            \sqrt{N/2}W_{rr}, & r=s.
        \end{cases}
    \end{align*}
    Then $(\widehat\xi_{\mathbf{i}})_{\mathbf{i}\in\mathcal E_N}$ are independent standard Gaussian random variables and
    \begin{align*}
        \sigma_{\mathbf{i}}=\sigma_r\sigma_s.
    \end{align*}

    For $\mathbf{i}\in\mathcal E_N$, let $\boldsymbol\zeta^{(\mathbf{i})}$ denote the collection of all coordinates other than the one indexed by $\mathbf{i}$, 
    defined by
    \begin{align*}
        \zeta_{\mathbf{j}}^{(\mathbf{i})}=\begin{cases}
                        \xi_{\mathbf{j}}, & \mathbf{j}\prec\mathbf{i},\\
                        \widehat\xi_{\mathbf{j}}, & \mathbf{i}\prec\mathbf{j},
                      \end{cases}
        \qquad \mathbf{j}\in\mathcal E_N\setminus\{\mathbf{i}\}.
    \end{align*}

    Thus $\boldsymbol\zeta^{(\mathbf{i})}$ is independent of both $\xi_{\mathbf{i}}$ and $\widehat\xi_{\mathbf{i}}$.
    Condition on $\boldsymbol\zeta^{(\mathbf{i})}$, define
    \begin{align*}
        h_{\mathbf{i}}\left(u;\boldsymbol\zeta^{(\mathbf{i})}\right)=\varphi\left(s_N^{-1}\left[\log\left(2^{-N}\sum_{\sigma\in\Sigma_N}
        \exp\left\{\frac{\beta_N}{2\sqrt N}\left(\sum_{\mathbf{j}\in\mathcal E_N\setminus\{\mathbf{i}\}}a_{\mathbf{j}}\zeta_{\mathbf{j}}^{(\mathbf{i})}\sigma_{\mathbf{j}}+a_{\mathbf{i}}u\sigma_{\mathbf{i}}\right)\right\}
        \right)-a_N\right]\right).
    \end{align*}

    By the definition of $(\widehat\xi_{\mathbf{i}})_{\mathbf{i}\in\mathcal E_N}$,
    \begin{align*}
        H_N(\sigma)=\frac{1}{2\sqrt{N}}\sum_{\mathbf{i}=(r,s)\in\mathcal E_N}a_{\mathbf{i}}\widehat\xi_{\mathbf{i}}\sigma_r\sigma_s.
    \end{align*}
    Replacing the coordinates in the order $\prec$ gives
    \begin{align*}
        \mathbb E\varphi\left(\frac{F_{N,\beta_N}^{\boldsymbol\xi}-a_N}{s_N}\right)-\mathbb E\varphi\left(\frac{F_{N,\beta_N}-a_N}{s_N}\right)
        =\sum_{\mathbf{i}\in\mathcal E_N}\mathbb E_{\boldsymbol\zeta^{(\mathbf{i})}}
        \left[\mathbb E_{\xi_{\mathbf{i}}}h_{\mathbf{i}}\left(\xi_{\mathbf{i}};\boldsymbol\zeta^{(\mathbf{i})}\right)-\mathbb E_{\widehat\xi_{\mathbf{i}}}h_{\mathbf{i}}\left(\widehat\xi_{\mathbf{i}};\boldsymbol\zeta^{(\mathbf{i})}\right)\right].
    \end{align*}

    Put $\lambda_{\mathbf{i},N}=\beta_Na_{\mathbf{i}}/(2\sqrt N)$.  If $\mathbf{i}=(r,s)$ with $r<s$, define
    \begin{align*}
        c_{\mathbf{i},\pm}\left(\boldsymbol\zeta^{(\mathbf{i})}\right)=2^{-N}\sum_{\substack{\sigma\in\Sigma_N\\ \sigma_{\mathbf{i}}=\pm1}}
        \exp\left\{\frac{\beta_N}{2\sqrt N}\sum_{\mathbf{j}\in\mathcal E_N\setminus\{\mathbf{i}\}}a_{\mathbf{j}}\zeta_{\mathbf{j}}^{(\mathbf{i})}\sigma_{\mathbf{j}}\right\}
    \end{align*}
    and
    \begin{align*}
        r_{\mathbf{i}}\left(\boldsymbol\zeta^{(\mathbf{i})}\right)=\frac12\log\frac{c_{\mathbf{i},+}\left(\boldsymbol\zeta^{(\mathbf{i})}\right)}{c_{\mathbf{i},-}\left(\boldsymbol\zeta^{(\mathbf{i})}\right)}.
    \end{align*}
    Then the partition sum inside the logarithm in the definition of $h_{\mathbf{i}}$ is
    \begin{align*}
        Z_{\mathbf{i}}\left(u;\boldsymbol\zeta^{(\mathbf{i})}\right)=
        2\sqrt{c_{\mathbf{i},+}\left(\boldsymbol\zeta^{(\mathbf{i})}\right)c_{\mathbf{i},-}\left(\boldsymbol\zeta^{(\mathbf{i})}\right)}\cosh\left(\lambda_{\mathbf{i},N}u+r_{\mathbf{i}}\left(\boldsymbol\zeta^{(\mathbf{i})}\right)\right),
    \end{align*}
    and hence
    \begin{align*}
        \frac{d^m}{du^m}\log Z_{\mathbf{i}}\left(u;\boldsymbol\zeta^{(\mathbf{i})}\right)=
        \lambda_{\mathbf{i},N}^m(\log\cosh)^{(m)}\left(\lambda_{\mathbf{i},N}u+r_{\mathbf{i}}\left(\boldsymbol\zeta^{(\mathbf{i})}\right)\right),
        \qquad 1\leq m\leq5.
    \end{align*}

    If $\mathbf{i}=(r,r)$, then $\sigma_{\mathbf{i}}=1$.  Thus
    \begin{align*}
        Z_{\mathbf{i}}\left(u;\boldsymbol\zeta^{(\mathbf{i})}\right)=
        \left(2^{-N}\sum_{\sigma\in\Sigma_N}
        \exp\left\{\frac{\beta_N}{2\sqrt N}\sum_{\mathbf{j}\in\mathcal E_N\setminus\{\mathbf{i}\}}a_{\mathbf{j}}\zeta_{\mathbf{j}}^{(\mathbf{i})}\sigma_{\mathbf{j}}\right\}\right)e^{\lambda_{\mathbf{i},N}u},
    \end{align*}
    so
    \begin{align*}
        \frac{d}{du}\log Z_{\mathbf{i}}\left(u;\boldsymbol\zeta^{(\mathbf{i})}\right)=\lambda_{\mathbf{i},N},
        \qquad
        \frac{d^m}{du^m}\log Z_{\mathbf{i}}\left(u;\boldsymbol\zeta^{(\mathbf{i})}\right)=0,\qquad 2\leq m\leq5.
    \end{align*}

    Since $|\lambda_{\mathbf{i},N}|\leq BN^{-1/2}$ and the functions $(\log\cosh)^{(m)}$ are bounded on $\mathbb R$ for every $1\leq m\leq 5$, 
    \begin{align}\label{derivation bound}
        \left|\frac{d^m}{du^m}\log Z_{\mathbf{i}}\left(u;\boldsymbol\zeta^{(\mathbf{i})}\right)\right| \leq C_{m,B}N^{-m/2}, \qquad 1\leq m\leq5,
    \end{align}
    uniformly in $\mathbf{i}\in\mathcal E_N$, $u$, and $\boldsymbol\zeta^{(\mathbf{i})}$.  
    Directly,
    \begin{align*}
        \frac{d}{du}h_{\mathbf{i}}\left(u;\boldsymbol\zeta^{(\mathbf{i})}\right)= 
        \frac{\varphi'\left(\left(\log Z_{\mathbf{i}}\left(u;\boldsymbol\zeta^{(\mathbf{i})}\right)-a_N\right)/s_N\right)}{s_N}\frac{d}{du}\log Z_{\mathbf{i}}\left(u;\boldsymbol\zeta^{(\mathbf{i})}\right)
    \end{align*}
    and
    \begin{align*}
        \frac{d^2}{du^2}h_{\mathbf{i}}\left(u;\boldsymbol\zeta^{(\mathbf{i})}\right)&=
        \frac{\varphi''\left(\left(\log Z_{\mathbf{i}}\left(u;\boldsymbol\zeta^{(\mathbf{i})}\right)-a_N\right)/s_N\right)}{s_N^2}
        \left(\frac{d}{du}\log Z_{\mathbf{i}}\left(u;\boldsymbol\zeta^{(\mathbf{i})}\right)\right)^2\\
        &+\frac{\varphi'\left(\left(\log Z_{\mathbf{i}}\left(u;\boldsymbol\zeta^{(\mathbf{i})}\right)-a_N\right)/s_N\right)}{s_N}\frac{d^2}{du^2}\log Z_{\mathbf{i}}\left(u;\boldsymbol\zeta^{(\mathbf{i})}\right).
    \end{align*}

    Applying \eqref{derivation bound}, similarly we get that for $1\leq m\leq 5$, 
    \begin{align}\label{Lindeberg test function derivatives}
        \sup_{\mathbf{i}\in\mathcal E_N}\left\|\frac{d^m}{du^m}h_{\mathbf{i}}\left(u;\boldsymbol\zeta^{(\mathbf{i})}\right)\right\|_\infty
        \leq C_{m,B}N^{-m/2}\sum_{\ell=1}^{m}s_N^{-\ell}\|\varphi^{(\ell)}\|_\infty.
    \end{align}

    Taylor's integral remainder gives
    \begin{align*}
        h_{\mathbf{i}}\left(u;\boldsymbol\zeta^{(\mathbf{i})}\right)-\sum_{m=0}^{4}\frac{\partial_u^m h_{\mathbf{i}}\left(0;\boldsymbol\zeta^{(\mathbf{i})}\right)}{m!}u^m=
        \frac{u^4}{3!}\int_0^1(1-t)^3\left[\partial_u^4h_{\mathbf{i}}\left(tu;\boldsymbol\zeta^{(\mathbf{i})}\right)-\partial_u^4h_{\mathbf{i}}\left(0;\boldsymbol\zeta^{(\mathbf{i})}\right)\right]dt.
    \end{align*}

    Bounding the difference on the R.H.S by 
    \begin{align*}
        \left|\partial_u^4h_{\mathbf{i}}\left(tu;\boldsymbol\zeta^{(\mathbf{i})}\right)-\partial_u^4h_{\mathbf{i}}\left(0;\boldsymbol\zeta^{(\mathbf{i})}\right)\right|
        \leq
        \min\left\{2\|\partial_u^4h_{\mathbf{i}}(\,\cdot\,;\boldsymbol\zeta^{(\mathbf{i})})\|_\infty,|tu|\left\|\partial_u^5h_{\mathbf{i}}\left(\,\cdot\,;\boldsymbol\zeta^{(\mathbf{i})}\right)\right\|_\infty\right\},
    \end{align*}
    we have
    \begin{align*}
        \left|h_{\mathbf{i}}\left(u;\boldsymbol\zeta^{(\mathbf{i})}\right)-\sum_{m=0}^{4}\frac{\partial_u^m h_{\mathbf{i}}\left(0;\boldsymbol\zeta^{(\mathbf{i})}\right)}{m!}u^m\right|\leq
        C|u|^4\times\min\left\{\left\|\partial_u^4h_{\mathbf{i}}\left(\,\cdot\,;\boldsymbol\zeta^{(\mathbf{i})}\right)\right\|_\infty,
        |u|\left\|\partial_u^5h_{\mathbf{i}}\left(\,\cdot\,;\boldsymbol\zeta^{(\mathbf{i})}\right)\right\|_\infty\right\}.
    \end{align*}

    Applying \eqref{Lindeberg test function derivatives} with $m=4,5$ yields
    \begin{align*}
        \left|h_{\mathbf{i}}\left(u;\boldsymbol\zeta^{(\mathbf{i})}\right)-\sum_{m=0}^{4}\frac{\partial_u^m h_{\mathbf{i}}\left(0;\boldsymbol\zeta^{(\mathbf{i})}\right)}{m!}u^m \right| 
        \leq C_BN^{-2} \left(\sum_{\ell=1}^{5}s_N^{-\ell}\|\varphi^{(\ell)}\|_\infty\right)|u|^4\min\left\{1,\frac{|u|}{\sqrt N}\right\}.
    \end{align*}

    Since $\xi_{\mathbf{i}}$ and $\widehat\xi_{\mathbf{i}}$ have identical moments through order three,
    \begin{align*}
        \mathbb E_{\xi_{\mathbf{i}}}h_{\mathbf{i}}\left(\xi_{\mathbf{i}};\boldsymbol\zeta^{(\mathbf{i})}\right)-
        \mathbb E_{\widehat\xi_{\mathbf{i}}}h_{\mathbf{i}}\left(\widehat\xi_{\mathbf{i}};\boldsymbol\zeta^{(\mathbf{i})}\right)
        =\frac{\mathbb E[\xi_{\mathbf{i}}^4]-3}{4!}\partial_u^4h_{\mathbf{i}}\left(0;\boldsymbol\zeta^{(\mathbf{i})}\right)
        +\mathbb E_{\xi_{\mathbf{i}}}\mathcal R_{\mathbf{i}}\left(\xi_{\mathbf{i}};\boldsymbol\zeta^{(\mathbf{i})}\right)
        -\mathbb E_{\widehat\xi_{\mathbf{i}}}\mathcal R_{\mathbf{i}}\left(\widehat\xi_{\mathbf{i}};\boldsymbol\zeta^{(\mathbf{i})}\right)
        ,
    \end{align*}
    where
    \begin{align*}
        \mathcal R_{\mathbf{i}}\left(u;\boldsymbol\zeta^{(\mathbf{i})}\right)
        =h_{\mathbf{i}}\left(u;\boldsymbol\zeta^{(\mathbf{i})}\right)
        -\sum_{m=0}^{4}\frac{\partial_u^m h_{\mathbf{i}}\left(0;\boldsymbol\zeta^{(\mathbf{i})}\right)}{m!}u^m.
    \end{align*}

    Hence, for every realization of $\boldsymbol\zeta^{(\mathbf{i})}$,
    \begin{align*}
        \left|\mathbb E_{\xi_{\mathbf{i}}}h_{\mathbf{i}}\left(\xi_{\mathbf{i}};\boldsymbol\zeta^{(\mathbf{i})}\right)- 
        \mathbb E_{\widehat\xi_{\mathbf{i}}}h_{\mathbf{i}}\left(\widehat\xi_{\mathbf{i}};\boldsymbol\zeta^{(\mathbf{i})}\right)\right|
        &\leq C_BN^{-2}\left|\mathbb E[\xi_{\mathbf{i}}^4]-3\right|
        \sum_{\ell=1}^{4}s_N^{-\ell}\|\varphi^{(\ell)}\|_\infty
        \\
        &\quad+C_BN^{-2}\left(\tau_N(\xi_{\mathbf{i}})+\tau_N(\widehat\xi)\right)
        \sum_{\ell=1}^{5}s_N^{-\ell}\|\varphi^{(\ell)}\|_\infty.
    \end{align*}

    Since $\widehat\xi\sim\mathcal N(0,1)$,
    \begin{align*}
        \tau_N(\widehat\xi)\leq N^{-1/2}\mathbb E|\widehat\xi|^5\leq CN^{-1/2}.
    \end{align*}

    Thus the preceding bound becomes
    \begin{align*}
        \left|\mathbb E_{\xi_{\mathbf{i}}}h_{\mathbf{i}}\left(\xi_{\mathbf{i}};\boldsymbol\zeta^{(\mathbf{i})}\right)- 
        \mathbb E_{\widehat\xi_{\mathbf{i}}}h_{\mathbf{i}}\left(\widehat\xi_{\mathbf{i}};\boldsymbol\zeta^{(\mathbf{i})}\right)\right|
        &\leq C_BN^{-2}\left|\mathbb E[\xi_{\mathbf{i}}^4]-3\right|
        \sum_{\ell=1}^{4}s_N^{-\ell}\|\varphi^{(\ell)}\|_\infty
        \\
        &\quad+C_B\left(N^{-2}\tau_N(\xi_{\mathbf{i}})+N^{-5/2}\right)
        \sum_{\ell=1}^{5}s_N^{-\ell}\|\varphi^{(\ell)}\|_\infty.
    \end{align*}

    Taking absolute values in the preceding telescoping identity and using
    $|\mathbb E_{\boldsymbol\zeta^{(\mathbf{i})}}Y|\leq\mathbb E_{\boldsymbol\zeta^{(\mathbf{i})}}|Y|$, we obtain
    \begin{align*}
        \left|\mathbb E\varphi\left(\frac{F_{N,\beta_N}^{\boldsymbol\xi}-a_N}{s_N}\right)
        -\mathbb E\varphi\left(\frac{F_{N,\beta_N}-a_N}{s_N}\right)\right|
        &\leq C_B\left(\Xi_N+N^{-1/2}\right)
        \sum_{\ell=1}^{5}s_N^{-\ell}\|\varphi^{(\ell)}\|_\infty.
    \end{align*}

    This proves \eqref{four moment comparison estimate}.

    The Kolmogorov estimate \eqref{four moment Kolmogorov comparison} is the standard smoothing consequence of \eqref{four moment comparison estimate}. Namely, with $Y_{\boldsymbol\xi}=(F_{N,\beta_N}^{\boldsymbol\xi}-a_N)/s_N$ and $Y_W=(F_{N,\beta_N}-a_N)/s_N$, one approximates the indicators $(-\infty,x]$ from above and below by smooth cutoffs with transition width $\varepsilon$ and derivative bounds $O(\varepsilon^{-r})$, applies \eqref{four moment comparison estimate} to the two cutoffs, and uses the density bound on $Z$ to obtain
    \begin{align*}
        d_{\mathrm K}\left(Y_{\boldsymbol\xi},Z\right)
        \leq d_{\mathrm K}\left(Y_W,Z\right)+M\varepsilon+C_B\left(\Xi_N+N^{-1/2}\right)\sum_{r=1}^{5}(s_N\varepsilon)^{-r}.
    \end{align*}

    Take
    \begin{align*}
        \varepsilon=s_N^{-1}\max\left\{
            \left((\Xi_N+N^{-1/2})s_N\right)^{1/6},
            \left((\Xi_N+N^{-1/2})s_N\right)^{1/2}
        \right\}.
    \end{align*}
    This gives \eqref{four moment Kolmogorov comparison}.
\end{proof}

\section{\texorpdfstring{Further lemmas for the proof of Proposition~\ref{L2 comparison} for $b>0$}{Further lemmas for the proof of Proposition}}\label{sec:further-positive-lemmas}
For the rest of the paper, we write
\begin{align*}
    \mathcal{Q}_N=\left\{-1,-1+\frac2N,\ldots,1-\frac2N,1\right\},\quad p_N(q)=2^{-N}\binom{N}{N(1+q)/2},\qquad q\in\mathcal Q_N.
\end{align*}
\begin{proof}[Proof of Lemma~\ref{tail bound implies L2}]
    For fixed $x\in\Sigma_N$ and $q\in\mathcal Q_N$, the relation $R(x,y)=q$ means that $y$ agrees with $x$ at exactly $N(1+q)/2$ coordinates. 
    Hence
    \begin{align*}
        \#\{y\in\Sigma_N:R(x,y)=q\}=\binom{N}{N(1+q)/2}.
    \end{align*}

    For the discrete overlap values $q\in\mathcal Q_N$, the definitions of $X_{N,\beta}$ and $J_\beta(q)$ give
    \begin{align*}
        X_{N,\beta}=\frac{Z_{N,\beta}}{Z_{N,\beta}^{sph}},\qquad J_\beta(q)=\mathbb E\left[\frac{e^{\beta H_N(x)+\beta H_N(y)}}{(Z_{N,\beta}^{sph})^2}\right].
    \end{align*}
    
    thus
    \begin{equation}
        \begin{aligned}
            \mathbb{E}X_{N,\beta}^2
            &=\mathbb{E}\left[\frac{\left(2^{-N}\sum_{x\in\Sigma_N}e^{\beta H_N(x)}\right)^2}{(Z_{N,\beta}^{sph})^2}\right]\\
            &=4^{-N}\sum_{x,y\in\Sigma_N}\mathbb{E}\left[\frac{e^{\beta H_N(x)+\beta H_N(y)}}{(Z_{N,\beta}^{sph})^2}\right]\\
            &=4^{-N}\sum_{q\in\mathcal Q_N}\sum_{x\in\Sigma_N}\sum_{\substack{y\in\Sigma_N\\ R(x,y)=q}}J_{\beta}(q)\\
            &=4^{-N}\sum_{q\in\mathcal Q_N}2^N\binom{N}{N(1+q)/2}J_{\beta}(q)\\
            &=\sum_{q\in\mathcal{Q}_N}p_N(q)J_{\beta}(q).
        \end{aligned}
    \end{equation}

    On the other hand, two independent Haar points $x,y$ on $S_N$ have overlap density $\rho_N$, and hence
    \begin{align}
        \int_{-1}^{1}\rho_N(q)J_{\beta}(q)dq
        &=\mathbb{E}_{x,y}[J_{\beta}(R(x,y))]\notag\\
        &=\mathbb{E}_W[(Z_{N,\beta}^{sph})^{-2}\mathbb{E}_{x,y}[e^{\beta H_N(x)+\beta H_N(y)}]]\notag\\
        &=1.\label{continuous identity}
    \end{align}

    Finally, by rotational invariance of $W$, $\mathbb{E}X_{N,\beta}=1$. 
    Therefore, the following second-moment identity holds.\footnotemark
    \begin{align}\label{second moment quadrature identity} 
        \mathbb{E}[(X_{N,\beta}-1)^2]= \sum_{q\in\mathcal{Q}_N}p_N(q)J_{\beta}(q)-\int_{-1}^{1}\rho_N(q)J_{\beta}(q)dq.
    \end{align}
    \footnotetext{This identity appears in \cite{DuHuang}; we give a proof here for completeness.}

    We split the right side into a center part, a middle tail, and a far tail. 
    Let
    \begin{align*}
        q_*=\min\{q\in\mathcal{Q}_N:q\geq u_{\beta}\}.
    \end{align*}

    Then
    \begin{align*}
        u_{\beta}\leq q_*<u_{\beta}+\frac2N.
    \end{align*}

    In the following center comparison, $q\in\mathcal Q_N$ is a discrete overlap value, while $t$ is a continuous overlap variable.
    We need the following estimation of $J_{\beta}(q)$, which shows that $J_{\beta}(q)$ is log-Lip continuous and upper bounded.
    \begin{lemma}\label{K beta derivative bound}
        For $-1<q<1$,
        \begin{align}\label{J local derivative bound} 
            \left|\frac{d}{dq}\log J_{\beta}(q)\right| \leq N\beta^2\left|q\right|,
        \end{align}
        and for $-1\leq q\leq1$,
        \begin{align}\label{J rough upper bound}
            J_{\beta}(q)\leq e^{N\beta^2q^2/2}.
        \end{align}
    \end{lemma}

    The proof of Lemma~\ref{K beta derivative bound} is given in Section~\ref{sec:remaining-auxiliary-lemmas}.

    The intervals $(q-N^{-1},q+N^{-1})$, indexed by $q\in\mathcal Q_N$ with $\left|q\right|\leq q_*$, partition $\left(-q_*-N^{-1},q_*+N^{-1}\right)$ up to their endpoints, which form a Lebesgue-null set.
    We now compare each lattice term with the spherical integral in the neighbor.
    Lemma~3.9 of \cite{DuHuang} gives, uniformly for $q\in\mathcal Q_N$ with $\left|q\right|\leq q_*$ and
    $|t-q|\leq N^{-1}$,
    \begin{align}\label{cube sphere stirling center} 
        p_N(q)e^{Nq^2/2}=\frac2N\rho_N(t)e^{Nt^2/2}\left[1+O\left(N^{-1}+q^2+Nq^4\right)\right].
    \end{align}

    By \eqref{J local derivative bound}, for such $q,t$,
    \begin{align*}
        \left|\log\frac{J_{\beta}(q)}{J_{\beta}(t)}\right|
        \leq \int_{\min\{q,t\}}^{\max\{q,t\}}N\beta^2|s|ds
        \leq CN(|q|+N^{-1})|t-q|
        \leq C(|q|+N^{-1}).
    \end{align*}

    Since $Nq_*^4=o(1)$, this implies
    \begin{align*}
        J_{\beta}(q)=J_{\beta}(t)\left[1+O\left(|q|+N^{-1}\right)\right].
    \end{align*}

    Also
    \begin{align*}
        e^{N(t^2-q^2)/2}=1+O\left(|q|+N^{-1}\right),
    \end{align*}
    because $N|t^2-q^2|\leq C(|q|+N^{-1})$.

    Therefore, we have
    \begin{align}\label{one cell comparison detailed}
        p_N(q)J_{\beta}(q)
        &=\frac{N}{2}\int_{|t-q|\leq N^{-1}}p_N(q)J_{\beta}(q)dt\nonumber\\
        &=e^{-Nq^2/2}\int_{|t-q|\leq N^{-1}}\rho_N(t)J_{\beta}(q)e^{Nt^2/2}\left[1+O\left(N^{-1}+q^2+Nq^4\right)\right]dt\nonumber\\
        &=\int_{|t-q|\leq N^{-1}}\rho_N(t)J_{\beta}(t)\left[1+O\left(|q|+N^{-1}\right)\right]\left[1+O\left(N^{-1}+q^2+Nq^4\right)\right]dt\nonumber\\
        &=\left[1+O\left(N^{-1}+|q|+q^2+Nq^4\right)\right]\int_{q-N^{-1}}^{q+N^{-1}}\rho_N(t)J_{\beta}(t)dt.
    \end{align}

    {
    All the implicit constants in \eqref{cube sphere stirling center}--\eqref{one cell comparison detailed} are independent of $N$, $q$, and $b_N$.
    Since $b_N\geq b_0$,
    \begin{align*}
        q_*\leq L(\beta-1)+\frac2N\leq C(\beta-1).
    \end{align*}
    Hence summing the errors gives
    \begin{align}\label{center quadrature bound}
        \left| \sum_{\substack{q\in\mathcal Q_N\\ \left|q\right|\leq q_*}}p_N(q)J_{\beta}(q)-\int_{\left|t\right|\leq q_*+N^{-1}}\rho_N(t)J_{\beta}(t)dt\right|
        &\leq C\sum_{\substack{q\in\mathcal Q_N\\ \left|q\right|\leq q_*}}\left(N^{-1}+\left|q\right|+q^2+Nq^4\right)\int_{q-N^{-1}}^{q+N^{-1}}\rho_N(t)J_{\beta}(t)dt\notag\\
        &\leq C\left(N^{-1}+q_*+q_*^2+Nq_*^4\right)\int_{\left|t\right|\leq q_*+N^{-1}}\rho_N(t)J_{\beta}(t)dt\notag\\
        &\leq C\left(N^{-1}+q_*+q_*^2+Nq_*^4\right)\notag\\
        &\leq C\left(N^{-1}+(\beta-1)+(\beta-1)^2+N(\beta-1)^4\right).
    \end{align}
    }

    We next treat the middle ranges $q_*<\left|q\right|\leq v_{\beta}$: $q$ is continuous in the spherical estimates and belongs to $\mathcal Q_N$ in the lattice estimates.
    Since $q_*\geq u_{\beta}$, throughout this range, for large $N$,
    \begin{align}\label{Delta lower bounds middle range} 
        \delta=\beta\left|q\right|-\beta+1=\beta\left|q\right|-\frac{\beta u_{\beta}}{L}
        \geq\beta\left(1-\frac1L\right)\left|q\right|\geq\frac{L-1}{2}(\beta-1).
    \end{align}

    {
    Hence \eqref{main tail bound} gives the spherical middle tail estimate
    \begin{align}\label{spherical middle tail bound} 
        \int_{q_*+N^{-1}\leq\left|q\right|\leq v_{\beta}}\rho_N(q)J_{\beta}(q)dq
        \leq N^C\exp\{-cN(\beta-1)^3\}.
    \end{align}
    }

    We now pass from the spherical middle tail to the lattice middle tail. 
    The same Stirling expansion used in Lemma~3.9 of \cite{DuHuang} gives, 
    uniformly for $q\in\mathcal Q_N$ with $\left|q\right|\leq v_{\beta}=o(1)$,
    \begin{align}\label{cube sphere density ratio middle} 
        \frac{p_N(q)}{(2/N)\rho_N(q)}\leq C\exp\left\{N\left(\frac{q^4}{6}+Cq^6\right)\right\}\leq C\exp\{CNq^4\}.
    \end{align}

    Since $\left|q\right|\leq v_{\beta}=o(1)$, \eqref{Delta lower bounds middle range} implies
    \begin{align}\label{quartic absorbed by cubic middle} 
        q^4\leq C_Lv_{\beta}\delta^3=o(\delta^3)
    \end{align}
    uniformly in the middle range. 
    Combining \eqref{cube sphere density ratio middle} with \eqref{main tail bound}, and then using \eqref{quartic absorbed by cubic middle}, we get, for $q\in\mathcal Q_N$ with $q_*<\left|q\right|\leq v_{\beta}$,
    \begin{align*}
        p_N(q)J_{\beta}(q)&\leq \frac{C}{N}\exp\{CNq^4\}\rho_N(q)J_{\beta}(q)\\
        &\leq N^C\exp\{CNq^4-cN\delta^3\}\\
        &\leq N^C\exp\{-cN\delta^3\}.
    \end{align*}

    {
    Since $\mbox{Card}\{\mathcal{Q}_N\}\leq N+1$, we obtain
    \begin{align}\label{cube middle tail bound}
        \sum_{\substack{q\in\mathcal{Q}_N\\q_*<\left|q\right|\leq v_{\beta}}}p_N(q)J_{\beta}(q)
        \leq N^C\exp\{-cN(\beta-1)^3\}.
    \end{align}
    }

    Finally we handle $\left|q\right|\geq v_{\beta}$. 
    For the spherical tail, the definition of $\rho_N$ and \eqref{J rough upper bound} give, for $-1<q<1$,
    \begin{align}\label{spherical far tail pointwise}
        \rho_N(q)J_{\beta}(q)&\leq C\sqrt{N}\exp\left\{\frac{N-3}{2}\log(1-q^2)+\frac{N\beta^2q^2}{2}\right\}\nonumber\\
        &\leq C\sqrt{N}\exp\left\{C+\frac{N}{2}(\beta^2-1)q^2-\frac{N}{8}q^4\right\}.
    \end{align}

    In the last line we used $\log(1-x)\leq -x-x^2/2$ for $0\leq x<1$.
    {
    Since $\beta\to1$, for all sufficiently large $N$,
    \begin{align*}
        \beta^2-1\leq 3(\beta-1).
    \end{align*}
    }

    If $\left|q\right|\geq v_{\beta}=M\sqrt{\beta-1}$, then
    \begin{align}\label{positive quadratic absorbed sphere} 
        \frac{N}{2}(\beta^2-1)q^2\leq\frac{3N}{2}(\beta-1)q^2\leq\frac{3N}{2M^2}q^4.
    \end{align}

    Choosing $M$ large enough so that $3/(2M^2)\leq1/16$, \eqref{spherical far tail pointwise} and
    \eqref{positive quadratic absorbed sphere} imply
    \begin{align*}
        \rho_N(q)J_{\beta}(q) \leq C\sqrt{N}e^{-cNq^4},\qquad\left|q\right|\geq v_{\beta}.
    \end{align*}

    {
    Therefore
    \begin{align}\label{spherical far tail bound} 
        \int_{\left|q\right|\geq v_{\beta}}\rho_N(q)J_{\beta}(q)dq
        \leq C\sqrt{N}e^{-cNv_{\beta}^4}
        \leq N^C\exp\{-cN(\beta-1)^2\}.
    \end{align}
    }

    For $q\in\mathcal Q_N$, let
    \begin{align*}
        I(q)=\frac12\left[(1+q)\log(1+q)+(1-q)\log(1-q)\right].
    \end{align*}

    By the standard binomial entropy bound,
    \begin{align}\label{binomial entropy upper} 
        p_N(q)\leq Ce^{-NI(q)}\leq Ce^{-N\left(q^2/2+q^4/12\right)}.
    \end{align}

    Increasing $M$ if necessary so that $3/(2M^2)\leq1/24$, \eqref{J rough upper bound}, \eqref{binomial entropy upper} and \eqref{positive quadratic absorbed sphere} give, for $q\in\mathcal Q_N$ with $\left|q\right|\geq v_{\beta}$,
    \begin{align*}
        p_N(q)J_{\beta}(q)
        &\leq\exp\left\{-NI(q)+\frac{N\beta^2q^2}{2}\right\}\nonumber\\
        &\leq\exp\left\{\frac{N}{2}(\beta^2-1)q^2-\frac{N}{12}q^4\right\}\nonumber\\
        &\leq e^{-cNq^4}.
    \end{align*}

    {
    Since $\mbox{Card}\{\mathcal{Q}_N\}\leq N+1$, we obtain
    \begin{align}\label{cube far tail bound}
        \sum_{\substack{q\in\mathcal{Q}_N\\ \left|q\right|\geq v_{\beta}}}p_N(q)J_{\beta}(q)
        \leq CNe^{-cNv_{\beta}^4}
        \leq N^C\exp\{-cN(\beta-1)^2\}.
    \end{align}
    }

    {
    Combining \eqref{second moment quadrature identity}, \eqref{center quadrature bound},
    \eqref{spherical middle tail bound}, \eqref{cube middle tail bound}, \eqref{spherical far tail bound} and
    \eqref{cube far tail bound}, we get
    \begin{align*}
        \mathbb{E}[(X_{N,\beta}-1)^2]
        &\leq C\left(N^{-1}+(\beta-1)+(\beta-1)^2+N(\beta-1)^4\right)\notag\\
        &\quad+N^C\exp\{-cN(\beta-1)^3\}
        +N^C\exp\{-cN(\beta-1)^2\}.
    \end{align*}

    This proves \eqref{moderately supercritical L2 bound}.
    }
\end{proof}

\begin{proof}[Proof of Lemma~\ref{direct reduction to Delta F}]
    Since $J_{\beta}$ is even, it is enough to consider $0\leq q<1$. 
    In the application of this lemma, $u_{\beta}\leq|q|\leq v_{\beta}=o(1)$, so the endpoint values $q=\pm1$ do not arise.
    Write the spectral decomposition as 
    \begin{align*}
        W=O^T\Lambda O=O^T\operatorname{diag}(\lambda_1,\ldots,\lambda_N)O,\quad \lambda_1>\lambda_2>\cdots>\lambda_N.
    \end{align*}

    After adjoining independent Rademacher signs to the ordered eigenvectors, the resulting eigenbasis $O$ is Haar distributed and independent of $\Lambda$.
    We first condition on $\Lambda$ and put
    \begin{align*}
        u=\frac{O(x+y)}{\sqrt{2(1+q)}},\qquad v=\frac{O(x-y)}{\sqrt{2(1-q)}}.
    \end{align*}

    Since $x^Tx=y^Ty=N$ and $x^Ty=Nq$, we have $u,v\in S_N$ and $u^Tv=0$. By definition,
    \begin{align*}
        \beta H_N(x)+\beta H_N(y)=\frac{\beta}{2}\left[(1+q)u^T\Lambda u+(1-q)v^T\Lambda v\right].
    \end{align*}

    Since $O$ is Haar, the pair $(u,v)$ is uniformly distributed on
    \begin{align*}
        \{(u,v)\in S_N^2:u^Tv=0\}.
    \end{align*}

    In particular, $u$ has distribution $\nu_N$ and, conditional on a fixed $u\in S_N$, $v$ has the normalized Haar distribution on $S_N\cap u^{\perp}$; 
    denote this conditional measure by $\nu_u$.

    For $\theta>0$, let
    \begin{align*}
        Z_{\Lambda}(\theta)=\int_{S_N}\exp\left\{\frac{\theta}{2}z^T\Lambda z\right\}d\nu_N(z).
    \end{align*}

    Condition on $\Lambda$, this is precisely $Z_{N,\theta}^{sph}$. 
    Consequently,
    \begin{align}\label{conditional J exact integral} 
        J_{\beta}(q)=\mathbb{E}_{\Lambda}\left[\frac{I_{\Lambda}(q)}{Z_{\Lambda}(\beta)^2}\right],
    \end{align}
    where
    \begin{align*}
        I_{\Lambda}(q)=\int_{S_N}\exp\left\{\frac{\beta(1+q)}{2}u^T\Lambda u\right\}B_{\Lambda}(u)d\nu_N(u),
    \end{align*}
    and, for fixed $u\in S_N$,
    \begin{align*}
        B_{\Lambda}(u)=\int_{S_N\cap u^{\perp}}\exp\left\{\frac{\beta(1-q)}{2}v^T\Lambda v\right\}d\nu_u(v).
    \end{align*}

    Let $w$ be uniform on $S_N$. Relative to the fixed direction $u$, write
    \begin{align*}
        w=tu+\sqrt{1-t^2}\,v,\qquad -1\leq t\leq1,\quad v\in S_N\cap u^{\perp}.
    \end{align*}

    Then $v$ has distribution $\nu_u$, and $t$ has density
    \begin{align*}
        \rho_N(t)=\frac{\Gamma(N/2)}{\sqrt{\pi}\,\Gamma((N-1)/2)}(1-t^2)^{(N-3)/2}.
    \end{align*}

    Hence
    \begin{align}\label{full sphere slice decomposition} 
        Z_{\Lambda}(\beta(1-q))=\int_{-1}^1\rho_N(t)\int_{S_N\cap u^{\perp}}\exp\left\{\frac{\beta(1-q)}{2}w^T\Lambda w\right\}d\nu_u(v)dt.
    \end{align}

    We restrict the outer integral to $\left|t\right|\leq N^{-1}$. In this range,
    \begin{align}\label{vector norm bound}
        \left\|w-v\right\|\leq\sqrt{N}\left(|t|+\left|1-\sqrt{1-t^2}\right|\right)\leq\frac{C}{\sqrt{N}}.
    \end{align}

    Since $\Lambda$ is diagonal, let $w=(w_i)$ and $v=(v_i)$.  By \eqref{vector norm bound},
    \begin{align*}
        \left|w^T\Lambda w-v^T\Lambda v\right|&=\left|\sum_{i=1}^N\lambda_i(w_i^2-v_i^2)\right|\\
        &\leq \Vert\Lambda\Vert_{op}\sum_{i=1}^N\left|(w_i+v_i)(w_i-v_i)\right|\\
        &\leq \Vert\Lambda\Vert_{op}\left(\sum_{i=1}^N(w_i+v_i)^2\right)^{1/2}\left\|w-v\right\|\\
        &\leq C\Vert\Lambda\Vert_{op}.
    \end{align*}

    It follows that
    \begin{align*}
        \exp\left\{\frac{\beta(1-q)}{2}w^T\Lambda w\right\}\geq e^{-C\beta(1-q)\Vert\Lambda\Vert_{op}}\exp\left\{\frac{\beta(1-q)}{2}v^T\Lambda v\right\}.
    \end{align*}

    Moreover, the Gamma ratio estimate and $(1-t^2)^{(N-3)/2}\geq c$ on $\left|t\right|\leq N^{-1}$ give
    \begin{align*}
        \int_{-N^{-1}}^{N^{-1}}\rho_N(t)dt\geq cN^{-1/2}.
    \end{align*}

    Plugging these bounds into \eqref{full sphere slice decomposition} yields, uniformly in $u$,
    \begin{align}\label{slice upper bound} 
        B_{\Lambda}(u) \leq C\sqrt{N}\,e^{C\beta(1-q)\Vert\Lambda\Vert_{op}}Z_{\Lambda}(\beta(1-q)).
    \end{align}

    Inserting \eqref{slice upper bound} into $I_{\Lambda}(q)$, we obtain
    \begin{align*}
        I_{\Lambda}(q)\leq C\sqrt{N}\,e^{C\beta(1-q)\Vert\Lambda\Vert_{op}}Z_{\Lambda}(\beta(1-q))Z_{\Lambda}(\beta(1+q)).
    \end{align*}

    Since $\beta(1-q)$ is bounded in the range under consideration and $\Vert\Lambda\Vert_{op}=\Vert W\Vert_{op}$, 
    there is a constant $C_0>0$ such that \eqref{conditional J exact integral} gives
    \begin{align*}
        J_{\beta}(q)&\leq C\sqrt{N}\mathbb{E}\left[e^{C_0\Vert W\Vert_{op}}\frac{Z_{\Lambda}(\beta(1+q))Z_{\Lambda}(\beta(1-q))}{Z_{\Lambda}(\beta)^2}\right]\\
        &=C\sqrt{N}\,\mathbb{E}\exp\left\{F_{N,\beta(1+q)}^{sph}+F_{N,\beta(1-q)}^{sph}-2F_{N,\beta}^{sph}+C_0\Vert W\Vert_{op}\right\}.
    \end{align*}

    This is \eqref{J direct Delta F bound}.

\end{proof}

\begin{proof}[Proof of Lemma~\ref{Delta F exponential upper}]
    Both $\rho_N(q)$ and $\Delta F_N^{sph}(q)$ are even functions of $q$. 
    Thus it is enough to prove the claim for
    \begin{align*}
        u_{\beta}\leq q\leq v_{\beta},
    \end{align*}
    where $\delta=\beta q-\beta+1$.

    Let $\lambda_1\geq\cdots\geq\lambda_N$ be the eigenvalues of $W$. 
    Put
    \begin{align*}
        \rho_{sc}(x)&=\frac{1}{2\pi}\sqrt{(4-x^2)_+},\qquad
        m_{sc}(z)=\int_{-2}^{2}\frac{\rho_{sc}(x)}{z-x}\,dx,\quad z>2,\\
        l_{sc}(z)&=\int_{-2}^{2}\log(z-x)\rho_{sc}(x)\,dx,\qquad
        l_{sc}'(z)=m_{sc}(z),\\
        L_N(z)&=\frac1N\sum_{i=1}^N\log(z-\lambda_i),\qquad z>\lambda_1.
    \end{align*}

    For $\theta>0$, define
    \begin{align*}
        G_{\theta}(z)=\theta z-L_N(z), \quad z>\lambda_1.
    \end{align*}

    Put
    \begin{align*}
        \Phi_{\lambda_1}(\theta)=\inf_{z>\lambda_1}G_{\theta}(z).
    \end{align*}

    We apply \cite[Lemma~1.3]{BaikLee} to the one-replica spherical integral. 
    The parameter $\beta$ in \cite[Lemma~1.3]{BaikLee} is $\theta/2$ under our normalization $H_N(x)=x^TWx/2$. 
    Hence, for any $\gamma>\lambda_1$,
    \begin{align}\label{BL contour identity} 
        Z_N^{sph}(\theta) =\mathcal{C}_N\theta^{1-N/2}\int_{\gamma-i\infty}^{\gamma+i\infty}\exp\left\{\frac{N}{2}G_{\theta}(z)\right\}\frac{dz}{2\pi i},
    \end{align}
    where $\mathcal{C}_N$ is deterministic and does not depend on $\theta$ or $W$.

    By \cite[Lemma~4.1]{BaikLee}, there exists a unique $\gamma_{\theta}>\lambda_1$ satisfying
    \begin{align}\label{saddle point equation} 
        0=G_{\theta}'(\gamma_{\theta})=\theta-\frac1N\sum_{i=1}^N\frac{1}{\gamma_{\theta}-\lambda_i},
    \end{align}
    and
    \begin{align*}
        G_{\theta}(\gamma_{\theta})=\Phi_{\lambda_1}(\theta).
    \end{align*}

    We now use the same decomposition as in the proof of \cite[Proposition~4.4]{DuHuang}: 
    choose the vertical contour in \eqref{BL contour identity} to pass through $\gamma_{\theta}$, 
    subtract and add $G_{\theta}(\gamma_{\theta})$ inside the exponent, and put
    \begin{align*}
        I_{\theta}=\int_{\gamma_{\theta}-i\infty}^{\gamma_{\theta}+i\infty}\exp\left\{\frac{N}{2}\left(G_{\theta}(z)-G_{\theta}(\gamma_{\theta})\right)\right\}\frac{dz}{2\pi i}.
    \end{align*}

    Since $Z_{N,\theta}^{sph}>0$, $\mathcal C_N>0$, and $\theta>0$, \eqref{BL contour identity} implies $I_\theta>0$.

    Therefore
    \begin{align}\label{BL one replica decomposition} 
        F_{N,\theta}^{sph}=\log\mathcal{C}_N+\left(1-\frac{N}{2}\right)\log\theta+\frac{N}{2}\Phi_{\lambda_1}(\theta)+\log I_{\theta}.
    \end{align}

    We take the difference of \eqref{BL one replica decomposition} at $\theta=\beta(1+q),\beta(1-q),\beta$. 
    The constant terms cancel, and
    \begin{align}\label{Delta F from BL}
        \Delta F_N^{sph}(q)&=\left(1-\frac{N}{2}\right)\log(1-q^2)
        +\frac{N}{2}\bigg[\Phi_{\lambda_1}(\beta(1+q))+\Phi_{\lambda_1}(\beta(1-q))-2\Phi_{\lambda_1}(\beta)\bigg]\nonumber\\
        &\quad+\log I_{\beta(1+q)}+\log I_{\beta(1-q)}-2\log I_{\beta}.
    \end{align}

    We abbreviate the empirical and contour three-temperature differences in \eqref{Delta F from BL} by
    \begin{align*}
        \Delta\Phi_{\lambda_1}(q)&=\Phi_{\lambda_1}(\beta(1+q))+\Phi_{\lambda_1}(\beta(1-q))-2\Phi_{\lambda_1}(\beta),\\
        \mathcal{I}_N(q)&=\log I_{\beta(1+q)}+\log I_{\beta(1-q)}-2\log I_{\beta}.
    \end{align*}

    With this notation, \eqref{Delta F from BL} becomes
    \begin{align}\label{Delta F compact decomposition} 
        \Delta F_N^{sph}(q)=\left(1-\frac{N}{2}\right)\log(1-q^2)+\frac{N}{2}\Delta\Phi_{\lambda_1}(q)+\mathcal{I}_N(q).
    \end{align}

    From now on in this proof, put $x_N=(\lambda_1-2)_+$.

    Corresponding to the definition of $\Phi_{\lambda_1}(\theta)$ for the GOE matrix $W$, for every $x\geq2$, write
    \begin{align*}
        \phi_x(\theta)=\inf_{z\geq x}\{\theta z-l_{sc}(z)\}.
    \end{align*}

    Define the corresponding continuous three-temperature difference by
    \begin{align*}
        \Delta\phi_x(q)=\phi_x(\beta(1+q))+\phi_x(\beta(1-q))-2\phi_x(\beta).
    \end{align*}

    Inserting the semicircle infimum differences at $x_N$ and at $0$ into \eqref{Delta F compact decomposition}, we obtain
    \begin{align}\label{Delta F full decomposition}
        \Delta F_N^{sph}(q)&=\left(1-\frac{N}{2}\right)\log(1-q^2)+\frac{N}{2}\left(\Delta\Phi_{\lambda_1}(q)-\Delta\phi_{2+x_N}(q)\right)\\
        &\quad+\frac{N}{2}\left(\Delta\phi_{2+x_N}(q)-\Delta\phi_2(q)\right)+\frac{N}{2}\Delta\phi_2(q)+\mathcal{I}_N(q).\nonumber
    \end{align}

    Given \eqref{Delta F full decomposition}, we first evaluate the deterministic semicircle contribution.
    For $z\geq2$, the semicircle logarithmic potential has the explicit expression
    \begin{align*}
        l_{sc}(z)=\frac{z^2-z\sqrt{z^2-4}}{4}+\log\left(\frac{z+\sqrt{z^2-4}}{2}\right)-\frac12.
    \end{align*}

    Indeed, differentiating the right-hand side gives the Stieltjes transform of the semicircle law, 
    and both sides are asymptotic to $\log z$ as $z\to\infty$. 
    Consequently,
    \begin{align}\label{explicit expression of m_sc}
        \frac{d}{dz}\{\theta z-l_{sc}(z)\}=\theta-\frac{z-\sqrt{z^2-4}}{2}.
    \end{align}

    The second term on the right decreases from $1$ to $0$ as $z$ runs from $2$ to infinity. 
    Hence the minimizer in the definition of $\phi_2(\theta)$ is $\theta+\theta^{-1}$ for $0<\theta\leq1$, and is $2$ for $\theta\geq1$. 
    Substituting these two minimizers into the preceding explicit formula gives
    \begin{align}\label{semicircle Phi zero formula}
        \phi_2(\theta)=\begin{cases}
                            1+\theta^2/2+\log\theta,&0<\theta\leq1,\\[1mm]
                            2\theta-1/2,&\theta\geq1.
                       \end{cases}
    \end{align}

    In the present range,
    \begin{align*}
        \beta(1-q)=1-\delta,\qquad\beta(1+q)=2\beta-1+\delta.
    \end{align*}

    Thus $\beta(1-q)<1<\beta<\beta(1+q)$, and \eqref{semicircle Phi zero formula} gives
    \begin{align}\label{semicircle deterministic G0}
        \Delta\phi_2(q)&=\phi_2(\beta(1+q))+\phi_2(\beta(1-q))-2\phi_2(\beta)\nonumber\\
        &=\left(2\beta(1+q)-\frac12\right)+\left(1+\frac{(1-\delta)^2}{2}+\log(1-\delta)\right)-2\left(2\beta-\frac12\right)\nonumber\\
        &=\log(1-\delta)+\delta+\frac{\delta^2}{2}.
    \end{align}

    The logarithmic term in \eqref{Delta F full decomposition} is canceled by the spherical overlap density up to a polynomial factor. 
    Indeed, writing $\mathfrak c_N=\Gamma(N/2)/(\sqrt\pi\,\Gamma((N-1)/2))\leq C\sqrt N$, 
    and using $q\leq v_\beta=o(1)$, we have
    \begin{align}\label{rho log theta cancellation} 
        \rho_N(q)\exp\left\{\left(1-\frac{N}{2}\right)\log(1-q^2)\right\}=\mathfrak c_N(1-q^2)^{-1/2}\leq N^C.
    \end{align}

    To estimate the second and third terms in \eqref{Delta F full decomposition}, set
    \begin{align}\label{A and R definitions in main proof} 
        \mathcal{A}_N(q) =\left[\frac{N}{2}\left(\Delta\Phi_{\lambda_1}(q)-\Delta\phi_{2+x_N}(q)\right)\right]_+,
        \quad\mathcal{R}_N(x)=\frac{N}{2}\left(\Delta\phi_{2+x}(q)-\Delta\phi_2(q)\right).
    \end{align}

    Combining \eqref{Delta F full decomposition}, \eqref{semicircle deterministic G0}, and \eqref{rho log theta cancellation},
    we obtain
    \begin{align}\label{Delta F four block upper}
        \rho_N(q)\mathbb{E}\exp\{\Delta F_N^{sph}(q)+C_0\Vert W\Vert_{op}\}
        &\leq N^C\exp\left\{-\frac{N}{2}\left[-\log(1-\delta)-\delta-\frac{\delta^2}{2}\right]\right\}\nonumber\\
        &\quad\times\mathbb{E}\exp\{\mathcal{A}_N(q)+\mathcal{R}_N(x_N)+\mathcal{I}_N(q)+C_0\Vert W\Vert_{op}\}.
    \end{align}

    In \eqref{Delta F four block upper}, 
    the expectation contains four random contributions: the empirical infimum error, 
    the edge response of semicircle distribution at $x_N$, the centered contour factor, and the operator norm. 
    We control them by the following four lemmas.

    {
    In the present regime,
    \begin{align*}
        0<\delta\leq \beta v_\beta-\beta+1=O\left(\sqrt{\beta-1}\right)=o(1).
    \end{align*}
    Thus all constants below are independent of $b_N$ for all sufficiently large $N$.
    }

    \begin{lemma}\label{A upper bound main proof}
        For any $p>0$ and any $\epsilon>0$, there exists $C_{p,\epsilon}>0$ such that, uniformly for
        $u_{\beta}\leq q\leq v_{\beta}$,
        \begin{align*}
            \mathbb{E}e^{p\mathcal{A}_N(q)}\leq N^{C_{p,\epsilon}}\exp\left\{\frac{p\epsilon}{2}N\delta^3\right\}.
        \end{align*}
    \end{lemma}

    \begin{lemma}\label{R upper bound main proof}
        There exists $c_0>0$ such that, for any fixed $r>1$, there exists $C_r>0$ satisfying
        \begin{align*}
            \mathbb{E}e^{r\mathcal{R}_N(x_N)}\leq N^{C_r}\exp\left\{N\delta^3\left[\frac{r^3}{6(r+3c_0)^2}+o(1)\right]\right\}.
        \end{align*}

        Here $o(1)$ is uniform for $u_{\beta}\leq q\leq v_{\beta}$.
    \end{lemma}

    \begin{lemma}\label{density upper bound main proof}
        For any $p>0$, there exists $C_p>0$ such that, uniformly for $u_{\beta}\leq q\leq v_{\beta}$,
        \begin{align*}
            \mathbb{E}\exp\{p\mathcal{I}_N(q)\}\leq N^{C_p}.
        \end{align*}
    \end{lemma}

    \begin{lemma}\label{op upper bound main proof}
        For any $p>0$, there exists $C_p>0$ such that
        \begin{align*}
            \sup_{N\geq2}\mathbb{E}e^{p\Vert W\Vert_{op}}\leq C_p.
        \end{align*}
    \end{lemma}

    The four lemmas above are proved in Section~\ref{sec:remaining-auxiliary-lemmas}; the random-matrix input lemmas used in those proofs are proved in Section~\ref{sec:random-matrix-input-proofs}.

    {
    For all sufficiently large $N$,
    \begin{align}\label{deterministic barrier in lemma proof} 
        -\log(1-\delta)-\delta-\frac{\delta^2}{2}\geq\frac{\delta^3}{3}.
    \end{align}
    }

    By H\"older inequality,
    \begin{equation}
        \begin{aligned}\label{four random factors in lemma proof}
            \mathbb{E}\exp\{\mathcal{A}_N(q)+\mathcal{R}_N(x_N)+\mathcal{I}_N(q)+C_0\Vert W\Vert_{op}\}
            &\leq\left(\mathbb{E}e^{4\mathcal{A}_N(q)}\mathbb{E}e^{4\mathcal{R}_N(x_N)}\mathbb{E}e^{4\mathcal{I}_N(q)}\mathbb{E}e^{4C_0\Vert W\Vert_{op}}\right)^{1/4}\\
            &\leq N^C\exp\left\{\left(\frac{\epsilon}{2}+\frac{16}{6(4+3c_0)^2}+o(1)\right)N\delta^3\right\}.
        \end{aligned}
    \end{equation}

    We choose $\epsilon>0$ so small that
    \begin{align*}
        \frac{1}{6}-\frac{\epsilon}{2}-\frac{16}{6(4+3c_0)^2}>0.
    \end{align*}

    Combining \eqref{Delta F four block upper}, \eqref{deterministic barrier in lemma proof} and
    \eqref{four random factors in lemma proof} proves \eqref{target Delta F op bound}.
\end{proof}

\section{Proofs of the remaining auxiliary lemmas}\label{sec:remaining-auxiliary-lemmas}

\begin{proof}[Proof of Lemma~\ref{K beta derivative bound}]
    Write $(A,B)_F=\operatorname{Tr}(A^TB)$ and $\|A\|_F=(A,A)_F^{1/2}$ for the Frobenius inner product and norm on the space of real symmetric matrices.  In the matrix integrals below, $dM$ denotes Lebesgue measure on this space.  For a symmetric matrix $M$, write $Z_{N,\beta}^{sph}(M)$ for the spherical partition function obtained by replacing $W$ with $M$. 
    For the fixed pair $x,y\in S_N$ with $x^Ty=Nq$, put
    \begin{align*}
        A_{\beta,x,y}=\frac{\beta}{N}\left(xx^T+yy^T\right).
    \end{align*}

    Then
    \begin{align*}
        \beta H_N(x)+\beta H_N(y)=\frac{N}{2}(W,A_{\beta,x,y})_F,
    \end{align*}
    and
    \begin{align*}
        \left\Vert A_{\beta,x,y}\right\Vert_F^2&=\frac{\beta^2}{N^2}\operatorname{Tr}\left(xx^T+yy^T\right)^2\\
        &=\frac{\beta^2}{N^2}\left(\left\Vert x\right\Vert^4+\left\Vert y\right\Vert^4+2(x^Ty)^2\right)\\
        &=2\beta^2(1+q^2).
    \end{align*}

    By the Cameron--Martin identity for the GOE, we therefore have
    \begin{align*}
        J_{\beta}(q)
        &=\exp\left\{\frac{N}{4}\left\Vert A_{\beta,x,y}\right\Vert_F^2\right\}\mathbb{E}\left[Z_{N,\beta}^{sph}\left(W+A_{\beta,x,y}\right)^{-2}\right]\\
        &=e^{N\beta^2(1+q^2)/2}K_{\beta}(q),
    \end{align*}
    where
    \begin{align*}
        K_{\beta}(q)=\mathbb{E}\left[Z_{N,\beta}^{sph}\left(W+A_{\beta,x,y}\right)^{-2}\right].
    \end{align*}

    The two nonzero eigenvalues of $A_{\beta,x,y}$ are $\beta(1+q)$ and $\beta(1-q)$, with corresponding orthonormal eigenvectors
    \begin{align*}
        e_1=\frac{x+y}{\sqrt{2N(1+q)}},\qquad e_2=\frac{x-y}{\sqrt{2N(1-q)}}.
    \end{align*}

    Extend $e_1,e_2$ to an orthonormal basis of $\mathbb{R}^N$, and use this basis as the coordinate basis. 
    The spectral decomposition of $A_{\beta,x,y}$ is then
    \begin{align*}
        A_{\beta,x,y} =\beta(1+q)e_1e_1^T+\beta(1-q)e_2e_2^T.
    \end{align*}

    In this basis, $W$ still has the GOE law by orthogonal invariance, 
    while the spherical partition function is unchanged under the same change of basis. 
    Consequently,
    \begin{align}\label{representation of K}
        K_{\beta}(q)=\mathbb{E}\left[Z_{N,\beta}^{sph}\left(W+\beta(1+q)e_1e_1^T+\beta(1-q)e_2e_2^T\right)^{-2}\right].
    \end{align}

    In \eqref{representation of K}, exchange coordinates $e_1$ and $e_2$ shows that $K_{\beta}$ is even function.
    Let
    \begin{align*}
        S=e_1e_1^T+e_2e_2^T,\quad D=e_1e_1^T-e_2e_2^T.
    \end{align*}

    Let $M=W+\beta(1+q)e_1e_1^T+\beta(1-q)e_2e_2^T$, $K_{\beta}$ can be written as
    \begin{align*}
        K_{\beta}(q)=\int Z_{N,\beta}^{sph}(M)^{-2}\exp\left\{-\frac{N}{4}\left\Vert M-\beta S-\beta qD\right\Vert_F^2\right\}dM
    \end{align*}
    up to a normalizing constant independent of $q$.    

    For symmetric matrices $M_0,M_1$ and $0\leq t\leq1$, H\"older's inequality gives
    \begin{align*}
        Z_{N,\beta}^{sph}\left((1-t)M_0+tM_1\right)
        &=\int_{S_N}e^{(1-t)\beta H_{M_0}(z)}e^{t\beta H_{M_1}(z)}d\nu_N(z)\\
        &\leq Z_{N,\beta}^{sph}(M_0)^{1-t}Z_{N,\beta}^{sph}(M_1)^t.
    \end{align*}
    Thus $Z_{N,\beta}^{sph}(M)^{-2}$ is log-concave in $M$, while the shifted Gaussian factor is jointly log-concave in $(M,q)$. 
    Their product is jointly log-concave; 
    the Pr\'ekopa--Leindler theorem therefore shows that $f_{\beta}=\log K_{\beta}$ is concave, so $f_{\beta}''\leq0$.
    Direct differentiation of the Gaussian factor gives
    \begin{align*}
        f_{\beta}''(q)=-N\beta^2+\frac{N^2\beta^2}{4}\operatorname{Var}_q\left((M-\beta S-\beta qD,D)_F\right)\geq-N\beta^2.
    \end{align*}

    Since $f_{\beta}$ is even,
    \begin{align*}
        -N\beta^2q\leq f_{\beta}'(q)\leq0,\qquad q\geq0.
    \end{align*}

    Consequently, for $q\geq0$,
    \begin{align*}
        0\leq\frac{d}{dq}\log J_{\beta}(q)=N\beta^2q+f_{\beta}'(q)\leq N\beta^2q.
    \end{align*}

    Since $J_{\beta}$ is even, this proves \eqref{J local derivative bound}. 
    Moreover, $K_{\beta}(q)\leq K_{\beta}(0)$ for $q\geq0$. The identity
    \begin{align*}
        \int_{-1}^{1}\rho_N(q)J_{\beta}(q)dq=1
    \end{align*}
    and the monotonicity in $\left|q\right|$ imply $J_{\beta}(0)\leq1$. 
    Therefore, for $q\geq0$,
    \begin{align*}
        J_{\beta}(q)=e^{N\beta^2q^2/2}e^{N\beta^2/2}K_{\beta}(q)\leq e^{N\beta^2q^2/2}e^{N\beta^2/2}K_{\beta}(0)
        =e^{N\beta^2q^2/2}J_{\beta}(0)\leq e^{N\beta^2q^2/2}.
    \end{align*}

    The case $q\leq0$ follows by symmetry.  The endpoint values in \eqref{J rough upper bound} follow by continuity of the defining expression for $J_{\beta}(q)$ as $q\to\pm1$.
\end{proof}

\begin{remark}
    The differentiations in the preceding proof can be passed through the integral defining $K_{\beta}(q)$.
    Indeed, for every symmetric matrix $M$,
    \begin{align*}
        Z_{N,\beta}^{sph}(M)\geq\exp\left\{-\frac{\beta N}{2}\Vert M\Vert_{op}\right\},
    \end{align*}
    and hence
    \begin{align*}
        Z_{N,\beta}^{sph}(M)^{-2}\leq e^{\beta N\Vert M\Vert_{op}}\leq e^{\beta N\Vert M\Vert_F}.
    \end{align*}

    In the representation of $K_{\beta}(q)$ above, the shifted GOE density is proportional to
    \begin{align*}
        \exp\left\{-\frac{N}{4}\left\Vert M-\beta S-\beta qD\right\Vert_F^2 \right\}.
    \end{align*}

    On every compact subinterval of $(-1,1)$, its first two $q$-derivatives are bounded by a polynomial in $\Vert M\Vert_F$ times this density. 
    The preceding bound on $Z_{N,\beta}^{sph}(M)^{-2}$ is therefore dominated by an integrable Gaussian function of $M$. 
    Thus dominated convergence justifies the differentiation under the integral sign.
\end{remark}

\begin{lemma}[Uniform counting input]\label{uniform counting input}
    Let $\lambda_1\geq\cdots\geq\lambda_N$ be the eigenvalues of the GOE
    matrix used in this paper, and let
    \begin{align*}
        \mathcal{N}_N(t)=\#\{i:\lambda_i\geq t\}.
    \end{align*}

    Then, for every fixed $s>0$, there exists $C_s>0$ such that
    \begin{align*}
        \mathbb{E}\exp\left\{s\sup_{t\in\mathbb{R}}\left|\mathcal{N}_N(t)-N\int_t^{\infty}\rho_{sc}(x)\,dx\right|\right\}\leq N^{C_s}.
    \end{align*}
\end{lemma}

The proof of Lemma~\ref{uniform counting input} is given in Section~\ref{sec:random-matrix-input-proofs}.

\begin{proof}[Proof of Lemma~\ref{A upper bound main proof}]
    It is enough to consider $0<\epsilon<1$. 
    Fix $u_{\beta}\leq q\leq v_{\beta}$, and set
    \begin{align*}
        a=2+x_N+\frac{\epsilon}{2}\delta^2.
    \end{align*}

    Let
    \begin{align*}
        \mathcal D_N=\sup_{t\in\mathbb R}\left|\#\{i:\lambda_i\ge t\}-N\int_t^\infty\rho_{sc}(x)\,dx\right|.
    \end{align*}

    We work on the event
    \begin{align}\label{good event}
        \mathcal E_\epsilon=\left\{\mathcal D_N\le\frac14N\left(\frac{\epsilon}{2}\delta^2\right)^{3/2}\right\}.
    \end{align}

    Define the constrained empirical variational function
    \begin{align*}
        \Phi_a(\theta)=\inf_{z\ge a}G_\theta(z),
    \end{align*}
    and 
    \begin{align*}
        \Delta\Phi_a(q)=\Phi_a(\beta(1+q))+\Phi_a(\beta(1-q))-2\Phi_a(\beta).
    \end{align*}

    We write $\mathcal A_N(q)$ into the following three-term decomposition:
    \begin{align}\label{A three term decomposition}
        \mathcal A_N(q)&=\left[\frac N2\left(\Delta\Phi_{\lambda_1}(q)-\Delta\phi_{2+x_N}(q)\right)\right]_+\nonumber\\
        &=\left[\frac N2\big[\Delta\Phi_{\lambda_1}(q)-\Delta\Phi_a(q)+\Delta\Phi_a(q)-\Delta\phi_a(q)+\Delta\phi_a(q)-\Delta\phi_{2+x_N}(q)\big]\right]_+.
    \end{align}

    We control these three terms by the following three lemmas.

    \begin{lemma}\label{constrained variational comparison}
        For $a=2+x_N+\frac{\epsilon}{2}\delta^2$,
        \begin{align*}
            \Delta\Phi_{\lambda_1}(q)\le\Delta\Phi_a(q).
        \end{align*}
    \end{lemma}

    \begin{lemma}\label{constrained short increment}
        On $\mathcal E_\epsilon$,
        \begin{align*}
            \frac N2\left(\Delta\Phi_a(q)-\Delta\phi_a(q)\right)\le\frac{\mathcal D_N}{2}\log\frac4\epsilon.
        \end{align*}
    \end{lemma}

    \begin{lemma}\label{semicircle boundary shift}
        For $a=2+x_N+\frac{\epsilon}{2}\delta^2$,
        \begin{align*}
            \frac N2\left(\Delta\phi_a(q)-\Delta\phi_{2+x_N}(q)\right)\le\frac{\epsilon}{4}N\delta^3.
        \end{align*}
    \end{lemma}

    The three lemmas above are proved in Subsection~\ref{subsec:A-auxiliary-lemma-proofs}.

    By \eqref{A three term decomposition} and Lemma~\ref{constrained variational comparison}, \ref{constrained short increment}, \ref{semicircle boundary shift},
    on $\mathcal E_\epsilon$ we obtain
    \begin{align}\label{A good event bound} 
        \mathcal A_N(q)\le\frac{\mathcal D_N}{2}\log\frac4\epsilon+\frac{\epsilon}{4}N\delta^3.
    \end{align}

    By \eqref{good event} and $2^{-3/2}\sqrt\epsilon\log(4/\epsilon)\leq 1$ for $0<\epsilon<1$,
    \begin{align*}
        \frac{\mathcal D_N}{2}\log\frac4\epsilon\leq\frac{\epsilon}{4}N\delta^3.
    \end{align*}

    Hence the right-hand side of \eqref{A good event bound} is at most $\epsilon N\delta^3/2$, and therefore
    \begin{align}\label{A good event exponential}
        \mathbb E\left[e^{p\mathcal A_N(q)}{\bf1}_{\mathcal E_\epsilon}\right]\le\exp\left\{\frac{p\epsilon}{2}N\delta^3\right\}.
    \end{align}

    It remains to control $\mathcal E_\epsilon^c$. 
    For every fixed $z_0>\lambda_1$,
    \begin{align*}
        \Phi_{\lambda_1}\big(\beta(1+q)\big)+\Phi_{\lambda_1}\big(\beta(1-q)\big)
        &\leq \beta(1+q)z_0-L_N(z_0)+\beta(1-q)z_0-L_N(z_0)\\
        &=2\beta z_0-2L_N(z_0).
    \end{align*}
    Taking the infimum over $z_0>\lambda_1$ yields $\Delta\Phi_{\lambda_1}(q)\leq0$. 
    The same argument for $\phi_{2+x_N}$ gives $\Delta\phi_{2+x_N}(q)\leq0$. 
    Moreover, by \eqref{A and R definitions in main proof},
    \begin{align*}
        \frac N2\left(\Delta\phi_{2+x_N}(q)-\Delta\phi_2(q)\right)=\mathcal R_N(x_N).
    \end{align*}
    By \eqref{R explicit formula}, $\mathcal R_N(0)=0$, and \eqref{R derivative bound} gives $\mathcal R_N'(x)\geq0$. 
    Hence $\mathcal R_N(x_N)\geq0$, and therefore $\Delta\phi_{2+x_N}(q)\geq\Delta\phi_2(q)$. 
    Hence, by \eqref{semicircle deterministic G0}, for all large $N$,
    \begin{align}\label{A deterministic truncation} 
        \mathcal A_N(q)\le-\frac N2\Delta\phi_{2+x_N}(q)\le\frac N2\left[-\log(1-\delta)-\delta-\frac{\delta^2}{2}\right]\le\frac13N\delta^3.
    \end{align}

    For every fixed $s>0$, Lemma~\ref{uniform counting input} and Markov's inequality yield
    \begin{align}\label{A bad event probability}
        \mathbb P(\mathcal E_\epsilon^c)\le N^{C_s}\exp\left\{-\frac{s}{4}\left(\frac{\epsilon}{2}\right)^{3/2}N\delta^3\right\}.
    \end{align}

    For $s>\frac{4p}{3}\left(\frac{2}{\epsilon}\right)^{3/2}$, \eqref{A good event exponential}, \eqref{A deterministic truncation}, and \eqref{A bad event probability} give
    \begin{align*}
        \mathbb E e^{p\mathcal A_N(q)}
        &=\mathbb E\left[e^{p\mathcal A_N(q)}{\bf1}_{\mathcal E_\epsilon}\right]
        +\mathbb E\left[e^{p\mathcal A_N(q)}{\bf1}_{\mathcal E_\epsilon^c}\right]\\
        &\leq \exp\left\{\frac{p\epsilon}{2}N\delta^3\right\}
        +\exp\left\{\frac{p}{3}N\delta^3\right\}\mathbb P(\mathcal E_\epsilon^c)\\
        &\leq \exp\left\{\frac{p\epsilon}{2}N\delta^3\right\}
        +N^{C_s}\exp\left\{\left[\frac p3-\frac{s}{4}\left(\frac{\epsilon}{2}\right)^{3/2}\right]N\delta^3\right\}\\
        &\leq N^{C_{p,\epsilon}}\exp\left\{\frac{p\epsilon}{2}N\delta^3\right\}.
    \end{align*}
\end{proof}

\begin{lemma}[Soft-edge right-tail input]\label{edge tail input}
    There exist constants $c_0,C,x_0>0$ such that, for all $0\leq x\leq x_0$,
    \begin{align*}
        \mathbb{P}(\lambda_1\geq2+x)\leq Ce^{-c_0Nx^{3/2}}.
    \end{align*}
\end{lemma}

The proof of Lemma~\ref{edge tail input} is given in Section~\ref{sec:random-matrix-input-proofs}.

\begin{proof}[Proof of Lemma~\ref{R upper bound main proof}]
    {
    We work in the range $u_{\beta}\leq q\leq v_{\beta}$, so that $\delta=\beta q-\beta+1$ satisfies $0<\delta=o(1)$ uniformly in $b_N$.
    }

    For $\theta>1$ and $z\geq2+x$, by \eqref{explicit expression of m_sc}
    \begin{align*}
        \frac{d}{dz}\{\theta z-l_{sc}(z)\}=\theta-m_{sc}(z)\geq\theta-1>0,
    \end{align*}
    thus
    \begin{align}\label{semicircle upper boundary infima}
        \phi_{2+x}(\beta)&=\inf_{z\geq2+x}\{\beta z-l_{sc}(z)\}=\beta(2+x)-l_{sc}(2+x),\\
        \phi_{2+x}(\beta(1+q))&=\inf_{z\geq2+x}\{\beta(1+q)z-l_{sc}(z)\}=\beta(1+q)(2+x)-l_{sc}(2+x).\label{semicircle upper boundary infima 2}
    \end{align}

    The unconstrained minimizer of $\theta z-l_{sc}(z)$ for $\theta=\beta(1-q)=1-\delta$ is
    \begin{align}\label{semicircle lower unconstrained minimizer}
        1-\delta+\frac{1}{1-\delta}=2+\frac{\delta^2}{1-\delta}>2.
    \end{align}

    Thus, if $0\leq x\leq\delta^2/(1-\delta)$, the constraint is inactive for $1-\delta$, 
    and direct substitution into the definition of $\mathcal{R}_N(x)$ and \eqref{semicircle deterministic G0} gives
    \begin{align*}
        \frac{2}{N}\mathcal{R}_N(x)&=\bigg[\beta(1+q)(2+x)-l_{sc}(2+x)+1+\frac{(1-\delta)^2}{2}+\log(1-\delta)\\
        &\qquad-2\big\{\beta(2+x)-l_{sc}(2+x)\big\}\bigg]-\left[\log(1-\delta)+\delta+\frac{\delta^2}{2}\right]\\
        &=\delta\cdot x-\int_0^x\left(1-m_{sc}(2+s)\right)ds.
    \end{align*}

    If $x\geq\delta^2/(1-\delta)$, all three infima defining $\Delta\phi_{2+x}(q)$ are attained at the common boundary point $2+x$, and hence
    \begin{align*}
        \Delta\phi_{2+x}(q)&=\phi_{2+x}(\beta(1+q))+\phi_{2+x}(\beta(1-q))-2\phi_{2+x}(\beta)\\
        &=\big[\beta(1+q)+\beta(1-q)-2\beta\big](2+x)\\
        &=0.
    \end{align*}

    The two expressions agree at $x=\delta^2/(1-\delta)$. 
    Hence
    \begin{align}\label{R explicit formula}
        \mathcal{R}_N(x)=N\begin{cases}\displaystyle \frac{\delta\cdot x}{2}-B(x),
                                &0\leq x\leq\frac{\delta^2}{1-\delta},\\[3mm]
                                \displaystyle \frac{\delta}{2}\frac{\delta^2}{1-\delta}-B\left(\frac{\delta^2}{1-\delta}\right),
                                &x\geq\frac{\delta^2}{1-\delta},
                          \end{cases}
    \end{align}
    where
    \begin{align*}
        B(x)=\frac12\int_0^x\left(1-m_{sc}(2+s)\right)ds=\frac14\int_0^x\left(\sqrt{s(4+s)}-s\right)ds.
    \end{align*}

    For $0\leq x\leq\delta^2/(1-\delta)$,
    \begin{align}\label{ineq of B(x)} 
        0\leq 1-m_{sc}(2+x) =\frac{\sqrt{x(4+x)}-x}{2}\leq\delta.
    \end{align}

    By \eqref{R explicit formula} and \eqref{ineq of B(x)}, we obtain
    \begin{align}\label{R derivative bound} 
        0\leq\mathcal{R}_N'(x)=\frac{N}{2}\left[\delta-\left(1-m_{sc}(2+x)\right)\right]\mathbf{1}_{0\leq x\leq\frac{\delta^2}{1-\delta}}\leq\frac{N\delta}{2}.
    \end{align}

    Since $\mathcal{R}_N(0)=0$, $\mathcal{R}_N(x)$ is nonnegative and nondecreasing in $x$. 
    Finally,
    \begin{align}\label{B lower bound} 
        B(x)\geq\frac14\int_0^x(2\sqrt{s}-s)ds=\frac{x^{3/2}}3-\frac{x^2}{8},
    \end{align}

    By Lemma~\ref{edge tail input}, for $0\leq x\leq\delta^2/(1-\delta)=o(1)$,
    \begin{align*}
        \mathbb{P}(x_N\geq x)\leq Ce^{-c_0Nx^{3/2}}
    \end{align*}
    holds for all sufficiently large $N$.
    Since $\mathcal{R}_N(0)=0$, \eqref{R explicit formula} gives
    \begin{align}\label{R tail integration bound}
        \mathbb{E}e^{r\mathcal{R}_N(x_N)}&=1+r\int_0^{\delta^2/(1-\delta)}\mathcal{R}_N'(x)e^{r\mathcal{R}_N(x)}\mathbb{P}(x_N\geq x)dx\notag\\
        &\leq 1+C r\int_0^{\delta^2/(1-\delta)}\mathcal{R}_N'(x)\exp\{r\mathcal{R}_N(x)-c_0Nx^{3/2}\}dx.
    \end{align}

    Let $x=\delta^2y^2$ with $0\leq y\leq(1-\delta)^{-1/2}$, \eqref{R explicit formula} and \eqref{B lower bound} give
    \begin{align}\label{R exponent bound}
        r\mathcal{R}_N(x)-c_0Nx^{3/2}&\leq N\left[\frac{r\delta x}{2}-\left(\frac r3+c_0\right)x^{3/2}+\frac r8x^2\right]\notag\\
        &=N\delta^3\left[\frac r2y^2-\left(\frac r3+c_0\right)y^3+\frac{r\delta}{8}y^4\right]\notag\\
        &\leq N\delta^3\left[\max_{u\geq0}\left\{\frac r2u^2-\left(\frac r3+c_0\right)u^3\right\}+o(1)\right]\notag\\
        &=N\delta^3\left[\frac{r^3}{6(r+3c_0)^2}+o(1)\right]. 
    \end{align}

    Also,
    \begin{align}\label{R derivative integral bound} 
        r\int_0^{\delta^2/(1-\delta)}\mathcal{R}_N'(x)dx\leq\frac{rN\delta^3}{2(1-\delta)}\leq N^{C_r}.
    \end{align}

    Combining \eqref{R tail integration bound}, \eqref{R exponent bound}, and \eqref{R derivative integral bound}, we obtain
    \begin{align*}
        \mathbb{E}e^{r\mathcal{R}_N(x_N)}\leq N^{C_r}\exp\left\{N\delta^3 \left[\frac{r^3}{6(r+3c_0)^2}+o(1)\right]\right\}.
    \end{align*}
\end{proof}

\begin{proof}[Proof of Lemma~\ref{density upper bound main proof}]
    Throughout this proof, we condition on eigenvalues $\lambda_1,\cdots,\lambda_N$.
    For
    \begin{align*}
        \theta\in\{\beta(1+q),\,\beta(1-q),\,\beta\},
    \end{align*}
    let $\gamma_\theta$ be the saddle point, which only depends on $\lambda_1,\cdots,\lambda_N$ by \eqref{saddle point equation}. Additionally,
    \begin{align*}
        I_{\theta}=\int_{\gamma_\theta-i\infty}^{\gamma_\theta+i\infty}\exp\left\{\frac{N}{2}\left(G_\theta(z)-\Phi_{\lambda_1}(\theta)\right)\right\}\frac{dz}{2\pi i}
        =\int_{\mathbb R}\exp\left\{\frac{N}{2}\left(G_\theta(\gamma_\theta+it)-G_\theta(\gamma_\theta)\right)\right\}\frac{dt}{2\pi},
    \end{align*}
    while
    \begin{align*}
        G_\theta(\gamma_\theta+it)-G_\theta(\gamma_\theta)
        =i\theta t-\frac1N\sum_{i=1}^N\log\left(\frac{\gamma_\theta+it-\lambda_i}{\gamma_\theta-\lambda_i}\right)
        =i\theta t-\frac1N\sum_{i=1}^N\log\left(1+\frac{it}{\gamma_\theta-\lambda_i}\right).
    \end{align*}

    Set
    \begin{align*}
        w_i=(\gamma_\theta-\lambda_i)^{-1},\qquad 1\leq i\leq N.
    \end{align*}

    Hence
    \begin{align*}
        I_\theta&=\int_{\mathbb R}\exp\left\{\frac{iN\theta t}{2}-\frac12\sum_{i=1}^N\log(1+itw_i)\right\}\frac{dt}{2\pi}\\
        &=\int_{\mathbb R}e^{itN\theta/2}\prod_{i=1}^N(1+itw_i)^{-1/2}\frac{dt}{2\pi}.
    \end{align*}

    By \eqref{saddle point equation},
    \begin{align*}
        \sum_{i=1}^Nw_i=N\theta.
    \end{align*}

    Throughout the range under consideration, $1/2\leq\theta\leq2$ for all sufficiently large $N$.

    Let $\xi_1,\ldots,\xi_N$ be independent standard Gaussian variables and put
    \begin{align*}
        S_\theta=\sum_{i=1}^Nw_i\xi_i^2.
    \end{align*}

    Its characteristic function is
    \begin{align*}
        \mathbb{E}e^{isS_\theta}=\prod_{i=1}^N\mathbb{E}e^{isw_i\xi_i^2}=\prod_{i=1}^N(1-2isw_i)^{-1/2},
    \end{align*}
    which is integrable for $N$ large.

    If $f_\theta$ denotes the density of $S_\theta$, Fourier inversion and the change of variables $t=-2s$ give
    \begin{align}\label{density-Fourier}
        2f_\theta(N\theta)=\frac1\pi\int_{\mathbb{R}}e^{-isN\theta}\prod_{i=1}^N(1-2isw_i)^{-1/2}ds
        =\int_{\mathbb{R}}e^{itN\theta/2}\prod_{i=1}^N(1+itw_i)^{-1/2}\frac{dt}{2\pi}
        =I_\theta.
    \end{align}

    We prove the uniform bounds 
    \begin{align}\label{density-two-sided} 
        \frac cN\le f_\theta(N\theta)\le C(1+\|W\|_{\mathrm{op}}).
    \end{align}

    For the upper bound, order the weights as $w_1\ge w_2\ge\cdots\ge w_N$, and write
    \begin{align*}
        Y=w_1\xi_1^2+w_2\xi_2^2,\qquad T=\sum_{i=3}^Nw_i\xi_i^2.
    \end{align*}

    For $x>0$, the convolution formula gives
    \begin{align*}
        f_Y(x)&=\frac1{2\pi\sqrt{w_1w_2}}\int_0^x\frac{\exp\{-u/(2w_1)-(x-u)/(2w_2)\}}{\sqrt{u(x-u)}}\,du\\
        &\leq\frac1{2\pi\sqrt{w_1w_2}}\int_0^x\frac{du}{\sqrt{u(x-u)}}\\
        &=\frac1{2\sqrt{w_1w_2}}.
    \end{align*}

    Since $Y$ and $T$ are independent,
    \begin{align*}
        f_\theta(x)=(f_Y*\mathcal{L}(T))(x)=\mathbb{E}\left[f_Y(x-T)\right]\leq\|f_Y\|_\infty,\qquad x\in\mathbb{R}.
    \end{align*}

    Thus the first two weighted chi-square variables control the density of the full sum. The saddle point equation implies
    \begin{align*}
        w_1\ge\frac1N\sum_iw_i=\theta,
    \end{align*}
    and hence $\gamma_\theta-\lambda_1\le\theta^{-1}\le2$. 
    Also,
    \begin{align*}
        w_1^{-1}\leq w_2^{-1}=\gamma_\theta-\lambda_2 \le\gamma_\theta-\lambda_1+\lambda_1-\lambda_2 \le2+2\|W\|_{\mathrm{op}}.
    \end{align*}

    Consequently,
    \begin{align*}
        f_\theta(N\theta)\le\frac C{\sqrt{w_1w_2}}\le C(1+\|W\|_{\mathrm{op}}).
    \end{align*}

    For the lower bound, write
    \begin{align*}
        V=\sum_{i=2}^Nw_i(\xi_i^2-1),\qquad\sigma^2=2\sum_{i=2}^Nw_i^2,\qquad r=\frac{w_1}{\sigma}.
    \end{align*}

    Since
    \begin{align*}
        S_\theta=N\theta-w_1+V+w_1\xi_1^2,
    \end{align*}
    conditioning on $V$ gives
    \begin{align}\label{conditional-density}
        f_\theta(N\theta)=\mathbb{E}\left[f_{w_1\xi_1^2}(w_1-V)\mathbf{1}_{\{V<w_1\}}\right]
        =\mathbb E\left[\frac{\exp\{-(w_1-V)/(2w_1)\}}{\sqrt{2\pi w_1(w_1-V)}}\mathbf{1}_{\{V<w_1\}}\right].
    \end{align}

    Suppose first that $r\le r_0$, where $r_0>0$ is a sufficiently small numerical constant. 
    Let $\Phi$ denote the distribution function of a standard normal random variable. 
    Berry--Esseen theorem gives
    \begin{align*}
        \sup_x\left|\mathbb P(V/\sigma\le x)-\Phi(x)\right|\le C\frac{\sum_{i=2}^Nw_i^3}{(\sum_{i=2}^Nw_i^2)^{3/2}}\le Cr.
    \end{align*}

    thus for any fixed $K>0$,
    \begin{align*}
        \mathbb P(V\leq w_1/2)\ge \Phi(r/2)-Cr,\quad \mathbb{P}(V<-Kw_1)\leq\Phi(-Kr)+Cr.
    \end{align*}

    Let $K\geq 1/2$ be sufficiently large, then choose $r_0$ such that $Kr_0\leq 1$, then
    \begin{align*}
        \mathbb{P}(-Kw_1\leq V\leq w_1/2)\geq \int_{-Kr}^{r/2}\frac{e^{-x^2/2}}{\sqrt{2\pi}}dx-2Cr\geq \frac{(K+\frac{1}{2})e^{-1/2}r}{\sqrt{2\pi}}-2Cr\geq Cr.
    \end{align*}

    On this event the integrand in \eqref{conditional-density} is at least $C/w_1$, and hence $f_\theta(N\theta)\ge C/\sigma$.

    If $r>r_0$, then, for a sufficiently large fixed $K$, Chebyshev's inequality and Cantelli's inequality give
    \begin{align*}
        \mathbb P(-K\sigma\le V\le w_1/2)
        &\geq1-\mathbb P(V<-K\sigma)-\mathbb P(V>w_1/2)\\
        &\geq1-K^{-2}-\frac1{1+r_0^2/4}=:c_{r_0}>0.
    \end{align*}

    Thus on this event $f_\theta(N\theta)\ge c/w_1$. 
    Combining the two cases,
    \begin{align*}
        f_\theta(N\theta)\ge\frac c{\max\{w_1,\sigma\}}.
    \end{align*}

    Finally,
    \begin{align*}
        w_1\le N\theta,\qquad\sigma^2\le2w_1(N\theta-w_1)\le\frac{(N\theta)^2}{2}.
    \end{align*}

    Since $N\theta\le2N$, the lower bound in \eqref{density-two-sided} follows.

    Using \eqref{density-Fourier} and \eqref{density-two-sided},
    \begin{align*}
        \exp\{p\mathcal E_N(q)\}&=\left(\frac{I_{\beta(1+q)}I_{\beta(1-q)}}{I_\beta^2}\right)^p\\
        &\le C_pN^{2p}(1+\|W\|_{\mathrm{op}})^{2p}.
    \end{align*}

    Taking expectations and using Lemma~\ref{op upper bound main proof} proves the claim.
\end{proof}

\begin{proof}[Proof of Lemma~\ref{op upper bound main proof}]
    Write $W=(G+G^T)/\sqrt{2N}$, where $G$ has independent standard Gaussian entries. 
    Then $\Vert W\Vert_{op}\leq \sqrt{2/N}\Vert G\Vert_{op}$. 
    The Gaussian singular value tail in \cite{DavidsonSzarek} gives
    \begin{align*}
        \mathbb{P}(\Vert G\Vert_{op}\geq 2\sqrt N+t)\leq e^{-t^2/2},\qquad t\geq0.
    \end{align*}

    Hence
    \begin{align*}
        \mathbb{P}(\Vert W\Vert_{op}\geq 3+t)\leq C e^{-cNt^2},\qquad t\geq0.
    \end{align*}

    Integrating this tail gives, for every fixed $p>0$,
    \begin{align*}
        \mathbb{E}e^{p\Vert W\Vert_{op}}\leq e^{3p}+p\int_3^{\infty}e^{pt}\mathbb{P}(\Vert W\Vert_{op}\geq t)dt\leq C_p,
    \end{align*}
    uniformly in $N$.
\end{proof}

\subsection{Proofs of the three auxiliary lemmas in the proof of Lemma~\ref{A upper bound main proof}}\label{subsec:A-auxiliary-lemma-proofs}

\begin{proof}[Proof of Lemma~\ref{constrained variational comparison}]
    Since $a>\lambda_1$, let $\gamma_\theta$ be the unique minimizer of $G_\theta$ on $(\lambda_1,\infty)$, and put
    \begin{align*}
        d_a(\theta)=\Phi_a(\theta)-\Phi_{\lambda_1}(\theta)
        =\begin{cases}
            0,& a\leq\gamma_\theta,\\
            G_\theta(a)-G_\theta(\gamma_\theta),& a>\gamma_\theta.
         \end{cases}
    \end{align*}

    Differentiating the identity in \eqref{saddle point equation} with respect to $\theta$ gives
    \begin{align*}
        \gamma_\theta'=-\left[\frac1N\sum_{i=1}^N\frac{1}{(\gamma_\theta-\lambda_i)^2}\right]^{-1}<0.
    \end{align*}

    Consequently, at every point other than the transition point $\gamma_\theta=a$,
    \begin{align*}
        d_a'(\theta)&=\mathbf{1}_{\{\gamma_\theta<a\}}(a-\gamma_\theta-\gamma_{\theta}'G'_{\theta}(\gamma_{\theta}))=(a-\gamma_\theta)_+,\\
        d_a''(\theta)&={\bf1}_{\{\gamma_\theta<a\}}\left[\frac1N\sum_{i=1}^N\frac{1}{(\gamma_\theta-\lambda_i)^2}\right]^{-1}\geq0.
    \end{align*}

    At $\gamma_{\theta}$ the two expressions for $d_a$ and $d_a'$ agree, hence $d_a$ is convex, which gives
    \begin{align*}
        d_a\big(\beta(1+q)\big)+d_a\big(\beta(1-q)\big)-2d_a(\beta)\geq0.
    \end{align*}

    By the definitions of the two three-temperature differences,
    \begin{align*}
        \Delta\Phi_a(q)-\Delta\Phi_{\lambda_1}(q)=d_a\big(\beta(1+q)\big)+d_a\big(\beta(1-q)\big)-2d_a(\beta).
    \end{align*}

    This proves $\Delta\Phi_{\lambda_1}(q)\leq\Delta\Phi_a(q)$.
\end{proof}

\begin{proof}[Proof of Lemma~\ref{constrained short increment}]
    By definition of $\Delta\Phi_a$ and $\Delta\phi_a$,
    \begin{align}\label{Delta Phi-Delta phi}
        \Delta\Phi_a(q)-\Delta\phi_a(q)=
        &\Phi_a\big(\beta(1+q)\big)+\Phi_a\big(\beta(1-q)\big)-2\Phi_a(\beta)\nonumber\\
        &-\phi_a\big(\beta(1+q)\big)-\phi_a\big(\beta(1-q)\big)+2\phi_a(\beta).
    \end{align}

    Put $z_-=\max\left\{a,2+\frac{\delta^2}{1-\delta}\right\}$.
    By \eqref{semicircle upper boundary infima} and \eqref{semicircle upper boundary infima 2}, with $2+x=a$,
    \begin{align*}
        \phi_a(\beta)&=\beta a-l_{sc}(a),
        &\phi_a\big(\beta(1+q)\big)&=\beta(1+q)a-l_{sc}(a).
    \end{align*}

    By \eqref{semicircle lower unconstrained minimizer} and the definition of $z_-$,
    \begin{align*}
        \phi_a\big(\beta(1-q)\big)=\beta(1-q)z_--l_{sc}(z_-).
    \end{align*}

    To estimate the three empirical infima in \eqref{Delta Phi-Delta phi}, we use the following lemma.

    \begin{lemma}\label{short increment endpoint estimates}
        On $\mathcal E_\epsilon$,
        \begin{align}\label{short increment empirical upper infima}
            \Phi_a(\beta)=\beta a-L_N(a),
            \Phi_a\big(\beta(1+q)\big)=\beta(1+q)a-L_N(a).
        \end{align}

        Moreover,
        \begin{align}\label{short increment lower empirical upper} 
            \Phi_a\big(\beta(1-q)\big)\leq\beta(1-q)z_--L_N(z_-).
        \end{align}
    \end{lemma}

    The proof of Lemma~\ref{short increment endpoint estimates} is given in Subsection~\ref{subsec:A-auxiliary-lemma-proofs}.

    Consequently, the preceding semicircle identities and Lemma~\ref{short increment endpoint estimates} imply
    \begin{align}\label{short increment variational comparison}
        \Delta\Phi_a(q)-\Delta\phi_a(q)&\leq\big[\beta(1+q)a-L_N(a)+\beta(1-q)z_--L_N(z_-)-2\{\beta a-L_N(a)\}\big]\nonumber\\
        &\quad-\big[\beta(1+q)a-l_{sc}(a)+\beta(1-q)z_--l_{sc}(z_-)-2\{\beta a-l_{sc}(a)\}\big]\nonumber\\
        &=\big(L_N(a)-l_{sc}(a)\big)-\big(L_N(z_-)-l_{sc}(z_-)\big).
    \end{align}

    Let
    \begin{align*}
        \nu_N(dt)=\sum_{i=1}^N\delta_{\lambda_i}(dt)-N\rho_{sc}(t)dt,\qquad Q_N(t)=\nu_N([t,\infty)).
    \end{align*}

    Then $\nu_N(\mathbb R)=0$, $Q_N(t)=0$ for $t>2+x_N$, and $|Q_N(t)|\leq\mathcal D_N$. For $z\geq a$, put
    \begin{align*}
        f_z(t)=\log(z-t)-\log(a-t).
    \end{align*}

    Let $M>2+x_N$. Since $f_z(-\infty)=0$, $Q_N(M)=0$, and $Q_N(-R)=0$ for all sufficiently large $R$, integration by parts on $[-R,M]$, followed by $R\to\infty$, gives
    \begin{align*}
        N\left[\big(L_N(z)-l_{sc}(z)\big)-\big(L_N(a)-l_{sc}(a)\big)\right]
        =\int_{\mathbb R}f_z(t)\nu_N(dt)
        =\int_{-\infty}^{2+x_N}\left(\frac1{a-t}-\frac1{z-t}\right)Q_N(t)dt.
    \end{align*}

    Consequently,
    \begin{align}\label{short logarithmic potential increment} 
        N\left|\big(L_N(z)-l_{sc}(z)\big)-\big(L_N(a)-l_{sc}(a)\big)\right|\leq\mathcal D_N\log\frac{z-2-x_N}{(\epsilon/2)\delta^2}.
    \end{align}

    If $z_-=a$, the right-hand side of \,\eqref{short increment variational comparison} vanishes. 
    If $z_->a$, then
    \begin{align}\label{short increment endpoint ratio}
        1\leq\frac{z_--2-x_N}{a-2-x_N}=\frac{z_--2-x_N}{(\epsilon/2)\delta^2}\leq\frac{\delta^2/(1-\delta)}{(\epsilon/2)\delta^2}
        =\frac{2}{\epsilon(1-\delta)}\leq\frac4\epsilon,
    \end{align}
    for all sufficiently large $N$. 
    Taking $z=z_-$ in \eqref{short logarithmic potential increment}, and using \eqref{short increment variational comparison} and \eqref{short increment endpoint ratio}, gives
    \begin{align*}
        \frac N2\left(\Delta\Phi_a(q)-\Delta\phi_a(q)\right)\leq\frac{\mathcal D_N}{2}\log\frac4\epsilon.
    \end{align*}
\end{proof}

\begin{proof}[Proof of Lemma~\ref{semicircle boundary shift}]
    By the definition of \(\mathcal{R}_N\), and $a=2+x_N+\frac{\epsilon}{2}\delta^2$,
    \begin{align*}
        \frac N2\left(\Delta\phi_a(q)-\Delta\phi_{2+x_N}(q)\right)
        &=\frac N2\left[\left(\Delta\phi_a(q)-\Delta\phi_2(q)\right)-\left(\Delta\phi_{2+x_N}(q)-\Delta\phi_2(q)\right)\right]\\
        &=\mathcal R_N\left(x_N+\frac{\epsilon}{2}\delta^2\right)-\mathcal R_N(x_N)\\
        &=\int_{x_N}^{x_N+(\epsilon/2)\delta^2}\mathcal R_N'(x)dx\\
        &\leq\frac{\epsilon}{4}N\delta^3,
    \end{align*}
    where the integral estimate follows from \eqref{R explicit formula} and \eqref{R derivative bound}.
\end{proof}

\begin{proof}[Proof of Lemma~\ref{short increment endpoint estimates}]
    Let
    \begin{align}\label{short increment signed measure} 
        \nu_N(dt)=\sum_{i=1}^N\delta_{\lambda_i}(dt)-N\rho_{sc}(t)dt,\quad Q_N(t)=\nu_N([t,\infty)).
    \end{align}

    Thus
    \begin{align*}
        \mathcal D_N=\sup_{t\in\mathbb R}|Q_N(t)|,\qquad\mathcal E_\epsilon =\left\{\mathcal D_N\leq\frac14N\left(\frac{\epsilon}{2}\delta^2\right)^{3/2}\right\}.
    \end{align*}

    Moreover, $Q_N(t)=0$ for $t>2+x_N$ and $\nu_N(\mathbb R)=0$.

    Integrating by parts against \eqref{short increment signed measure} with $(a-t)^{-1}$ gives
    \begin{align*}
        \left|L_N'(a)-l_{sc}'(a)\right|\leq\frac{\mathcal D_N}{N}\int_{-\infty}^{2+x_N}\frac{dt}{(a-t)^2}=\frac{\mathcal D_N}{N(\epsilon/2)\delta^2}.
    \end{align*}

    Since $a=2+x_N+(\epsilon/2)\delta^2$, $x_N\geq0$, and $l_{sc}'=m_{sc}$ is decreasing on $(2,\infty)$, for all sufficiently large $N$,
    \begin{align*}
        l_{sc}'(a)\leq l_{sc}'\left(2+\frac{\epsilon}{2}\delta^2\right)
        =1-\frac{\sqrt{(\epsilon/2)\delta^2(4+(\epsilon/2)\delta^2)}-(\epsilon/2)\delta^2}{2}
        \leq 1-\frac12\sqrt{\frac{\epsilon}{2}}\delta.
    \end{align*}

    On $\mathcal E_\epsilon$,
    \begin{align}\label{short increment boundary slope}
        L_N'(a)\leq l_{sc}'(a)+\frac{\mathcal D_N}{N(\epsilon/2)\delta^2}
        \leq 1-\frac12\sqrt{\frac{\epsilon}{2}}\delta+\frac14\sqrt{\frac{\epsilon}{2}}\delta
        =1-\frac14\sqrt{\frac{\epsilon}{2}}\delta<1.
    \end{align}

    We have
    \begin{align*}
        L_N''(z)=-\frac1N\sum_{i=1}^N\frac1{(z-\lambda_i)^2}<0.
    \end{align*}

    Thus $L_N'$ is decreasing. 
    On $\mathcal E_\epsilon$, \eqref{short increment boundary slope} gives for $\theta\in\{\beta,\beta(1+q)\}$,
    \begin{align*}
        \frac{d}{dz}(\theta z-L_N(z))=\theta-L_N'(z)>0,\quad z\geq a,
    \end{align*}
    thus the two empirical infima in \eqref{short increment empirical upper infima} are attained at $a$. This proves \eqref{short increment empirical upper infima}.
    Since $z_-\geq a$, the definition of the empirical infimum gives
    \begin{align*}
        \Phi_a\big(\beta(1-q)\big)\leq\beta(1-q)z_--L_N(z_-),
    \end{align*}
    which is \eqref{short increment lower empirical upper}.
\end{proof}

\section{Proofs of the random-matrix input lemmas}\label{sec:random-matrix-input-proofs}

\begin{proof}[Proof of Lemma~\ref{uniform counting input}]
    We first work with the classical Gaussian ensembles whose GOE one-body weight is $e^{-x^2/2}$. 
    Let $M_m^{\mathrm R}\sim\mathrm{GOE}_m$ and $M_N^{\mathrm C}\sim\mathrm{GUE}_N$ denote the corresponding ensembles, 
    so their eigenvalue densities have one-body weights $e^{-x^2/2}$ and $e^{-x^2}$, respectively. 
    Thus
    \begin{align*}
        W\stackrel d=\sqrt{\frac2N}M_N^{\mathrm R}.
    \end{align*}

    Put $H=\sqrt{2/N}M_N^{\mathrm C}$. 
    We first prove a fixed-threshold estimate for $H$. 
    Put
    \begin{align*}
        C_N(t)=\#\{\lambda_i(H)\ge t\},\qquad\mu_{sc}(t)=\int_t^\infty\rho_{sc}(x)\,dx.
    \end{align*}

    The GUE eigenvalues form a finite-rank determinantal process. 
    Hence, by the Bernoulli decomposition for determinantal point processes \cite{HKPV},
    \begin{align*}
        C_N(t)\stackrel d=\sum_{j=1}^NB_j(t),
    \end{align*}
    where the $B_j(t)$'s are independent Bernoulli variables. 
    If their parameters are $\alpha_j(t)$, then for every fixed $u\in\mathbb R$,
    \begin{align}\label{GUE-Bernoulli-mgf}
        \log\mathbb E e^{u(C_N(t)-\mathbb EC_N(t))}&=\sum_j\left[\log\{1+\alpha_j(t)(e^u-1)\}-u\alpha_j(t)\right]\notag\\
        &\le c(u)\sum_j\alpha_j(t)(1-\alpha_j(t))\notag\\
        &=c(u)\operatorname{Var}C_N(t).                                  
    \end{align}

    Theorem~1.1 of \cite{GotzeTikhomirov} gives
    \begin{align}\label{GUE-count-mean}
        \sup_{t\in\mathbb R}|\mathbb EC_N(t)-N\mu_{sc}(t)|\le C.
    \end{align}

    By the symmetry of the GUE,
    \begin{align*}
        \operatorname{Var}C_N(-t)=\operatorname{Var}C_N(t).
    \end{align*}
    It is therefore enough to consider $t\ge0$.  We claim that there is a numerical constant $C$ such that
    \begin{align*}
        \sup_{t\ge0}\operatorname{Var}C_N(t)\le C\log N
    \end{align*}
    for all sufficiently large $N$.  Otherwise, there would exist a subsequence $N_k\to\infty$, thresholds $t_k\ge0$, and constants $C_k\to\infty$ such that
    \begin{align*}
        \operatorname{Var}C_{N_k}(t_k)>C_k\log N_k.
    \end{align*}
    The Bernoulli representation and \eqref{GUE-count-mean} give
    \begin{align*}
        \operatorname{Var}C_{N_k}(t_k)
        \leq\mathbb EC_{N_k}(t_k)
        \leq N_k\mu_{sc}(t_k)+C.
    \end{align*}
    Hence $t_k<2$ for all sufficiently large $k$.  If $N_k(2-t_k)^{3/2}$ were bounded along a further subsequence, then the semicircle-tail bound $\mu_{sc}(t)\leq C(2-t)^{3/2}$ for $0\leq t<2$ would give
    \begin{align*}
        \operatorname{Var}C_{N_k}(t_k)
        \leq C N_k(2-t_k)^{3/2}+C=O(1),
    \end{align*}
    again a contradiction.  We may therefore assume that
    \begin{align*}
        N_k(2-t_k)^{3/2}\longrightarrow\infty.
    \end{align*}
    Since $t_k/2\in[0,1)$ and
    \begin{align*}
        N_k(1-t_k/2)^{3/2}\longrightarrow\infty,
    \end{align*}
    Lemma~2.3 of \cite{Gustavsson}, after the normalization from its GUE normalization to $H$, gives
    \begin{align*}
        \frac{\operatorname{Var}C_{N_k}(t_k)}{\log N_k}
        =\frac{1+o(1)}{2\pi^2}\frac{\log\left(N_k(2-t_k)^{3/2}\right)}{\log N_k}
        \leq\frac{1+o(1)}{2\pi^2}\frac{\log(2^{3/2}N_k)}{\log N_k}
        \leq\frac1{2\pi^2}+o(1).
    \end{align*}
    Here the $o(1)$ is used only along the subsequence $(t_k)$, as permitted by Gustavsson's result.  This contradicts $C_k\to\infty$ and proves the claim.
    Thus
    \begin{align}\label{GUE-count-variance}
        \sup_{t\in\mathbb R}\operatorname{Var}C_N(t)\le C\log N.
    \end{align}

    Combining \eqref{GUE-Bernoulli-mgf}, \eqref{GUE-count-mean}, and \eqref{GUE-count-variance}, we obtain
    \begin{align}\label{GUE-fixed-t}
        \sup_{t\in\mathbb R}\mathbb E\exp\left\{ u(C_N(t)-N\mu_{sc}(t))\right\}\le N^{C_u}. 
    \end{align}

    We next transfer \eqref{GUE-fixed-t} to GOE. 
    Let $\widetilde M_{N+1}^{\mathrm R}$ be an independent copy of $M_{N+1}^{\mathrm R}$. 
    The Forrester--Rains identity in \cite[Theorem~5.2, equation~(5.9)]{ForresterRains} reads
    \begin{align*}
        \operatorname{even}\left(M_N^{\mathrm R}\cup\widetilde M_{N+1}^{\mathrm R}\right)\stackrel d=M_N^{\mathrm C}.
    \end{align*}

    This identity means that, if we take eigenvalues of $M_N^{\mathrm R}$ and $\widetilde M_{N+1}^{\mathrm R}$ in order to get a random point process with 
    $(2N+1)$ particles, then the $N$ even particles has the same distribution as eigenvalues of $M_N^{\mathrm C}$.

    Scaling every eigenvalue by $\sqrt{2/N}$, set
    \begin{align*}
        \widetilde W_{N+1}=\sqrt{\frac2N}\widetilde M_{N+1}^{\mathrm R}.
    \end{align*}

    For a fixed $t$, let $A(t)$ and $B(t)$ be the counts above $t$ for $W$ and $\widetilde W_{N+1}$, respectively, and let $C(t)$ be the count above $t$ for $H$.
    Ordering the superposed eigenvalues decreasingly gives
    \begin{align*}
        A(t)+B(t)=2C(t)+\varepsilon_t,\qquad\varepsilon_t\in\{0,1\}.
    \end{align*}

    Consequently,
    \begin{align*}
        (A-\mathbb EA)+(B-\mathbb EB)=2(C-\mathbb EC)+(\varepsilon_t-\mathbb E\varepsilon_t),
    \end{align*}
    and independence of $A$ and $B$ yields
    \begin{align*}
        \mathbb E e^{u(A-\mathbb EA)} \mathbb E e^{u(B-\mathbb EB)}\le e^{|u|}\mathbb E e^{2u(C-\mathbb EC)}.
    \end{align*}

    Since the second factor on the left is at least one by Jensen's inequality, \eqref{GUE-fixed-t} gives
    \begin{align}\label{GOE-centered-fixed-t}
        \sup_t\mathbb E e^{u(A(t)-\mathbb EA(t))}\le N^{C_u}. 
    \end{align}

    It remains to identify the centering. 
    Couple $\widetilde W_{N+1}$ with its upper-left $N\times N$ principal minor. 
    This minor has the same law as $W$, and Cauchy interlacing gives
    \begin{align*}
        |\mathbb EA(t)-\mathbb EB(t)|\le 1.
    \end{align*}

    Taking expectations in $A+B=2C+\varepsilon_t$ therefore shows
    \begin{align*}
        |\mathbb EA(t)-\mathbb EC(t)|\le 1.
    \end{align*}

    Together with \eqref{GUE-count-mean}, this gives
    \begin{align}\label{GOE-count-centering}
        \sup_t|\mathbb EA(t)-N\mu_{sc}(t)|\le C. 
    \end{align}

    Combining \eqref{GOE-centered-fixed-t} and \eqref{GOE-count-centering}, for every fixed $u\in\mathbb R$,
    \begin{align}\label{GOE-fixed-t}
        \sup_t\mathbb E\exp\{u(A(t)-N\mu_{sc}(t))\}\le N^{C_u}. 
    \end{align}

    Finally choose semicircle tail quantiles
    \begin{align*}
        2=t_0>t_1>\cdots>t_N=-2,\qquad N\mu_{sc}(t_j)=j,
    \end{align*}
    and write
    \begin{align*}
        Q_N(t)=\mathcal N_N(t)-N\mu_{sc}(t).
    \end{align*}

    Monotonicity of the two counting functions implies
    \begin{align*}
        \sup_{t\in\mathbb R}|Q_N(t)|\le 1+\max_{0\le j\le N}|Q_N(t_j)|.
    \end{align*}

    The same bound outside $[-2,2]$ follows from the endpoint counts.
    Therefore, by \eqref{GOE-fixed-t},
    \begin{align*}
        \mathbb E e^{s\sup_t|Q_N(t)|}
        &\le e^s\sum_{j=0}^N\mathbb E e^{s|Q_N(t_j)|}\\
        &\le e^s\sum_{j=0}^N\left(\mathbb E e^{sQ_N(t_j)}+\mathbb E e^{-sQ_N(t_j)}\right)
        \le N^{C_s}.
    \end{align*}

    This proves the lemma.
\end{proof}

\begin{proof}[Proof of Lemma~\ref{edge tail input}]
    Let $H=\sqrt N\,W$. Then $H$ is the $\beta$-Hermite ensemble with $\beta=1$ in the normalization of Ledoux--Rider. 
    Their right-tail estimate \cite[Theorem~1]{LedouxRider} gives, for $0\le\varepsilon\le1$,
    \begin{align*}
        \mathbb P\left(\lambda_{\max}(H)\ge2\sqrt N(1+\varepsilon)\right)\le C\exp\{-cN\varepsilon^{3/2}\}.
    \end{align*}

    Taking $\varepsilon=x/2$ proves, for $0\le x\le x_0$,
    \begin{align*}
        \mathbb P(\lambda_1(W)\ge2+x)\le C e^{-c_0Nx^{3/2}}.
    \end{align*}
\end{proof}

\section{\texorpdfstring{Proof of Proposition~\ref{L2 comparison} for $b<0$}{Proof of Proposition for b<0}}\label{sec:negative-window-proof}

\begin{proof}
    Let
    \begin{align*}
        \beta=1+bN^{-1/3}\sqrt{\log N},\qquad b<0,
    \end{align*}
    and put
    \begin{align*}
        \psi_N(q)=\rho_N(q)J_\beta(q),\qquad -1<q<1.
    \end{align*}

    Then $\psi_N$ is an even probability density. Put
    \begin{align*}
        \kappa_N=N(1-\beta^2)-3=(-2b+o(1))N^{2/3}\sqrt{\log N}>0.
    \end{align*}

    Define
    \begin{align*}
        \widetilde\psi_N(q)=e^{\kappa_Nq^2/2}\psi_N(q),\qquad -1<q<1.
    \end{align*}

    By \eqref{J local derivative bound}, for $0\le q<1$,
    \begin{align*}
        \frac{d}{dq}\log\widetilde\psi_N(q)
        &=-\frac{(N-3)q}{1-q^2}+\frac{d}{dq}\log J_\beta(q)+\kappa_Nq\\
        &\le-\frac{(N-3)q}{1-q^2}+N\beta^2q+\kappa_Nq\\
        &=-(N-3)\frac{q^3}{1-q^2}\le0.
    \end{align*}

    Thus $\widetilde\psi_N$ is nonincreasing on $[0,1)$. For $0<r<1$,
    \begin{align*}
        \frac12=\int_0^1e^{-\kappa_Nq^2/2}\widetilde\psi_N(q)\,dq
        \ge\widetilde\psi_N(r)\int_0^re^{-\kappa_Nq^2/2}\,dq,
    \end{align*}
    while
    \begin{align*}
        \frac12\int_{r\le|q|<1}\psi_N(q)\,dq&=\int_r^1e^{-\kappa_Nq^2/2}\widetilde\psi_N(q)\,dq
        \le\widetilde\psi_N(r)\int_r^1e^{-\kappa_Nq^2/2}\,dq.
    \end{align*}

    Therefore, for $\kappa_N^{-1/2}\le r<1$,
    \begin{align}\label{negative-tail}
        \int_{r\le|q|<1}\psi_N(q)\,dq
        \le\frac{\displaystyle\int_r^1e^{-\kappa_Nq^2/2}\,dq}{\displaystyle\int_0^re^{-\kappa_Nq^2/2}\,dq}
        \le Ce^{-\kappa_Nr^2/2}.                                           
    \end{align}

    \begin{align}\label{negative-moments}
        \int_{-1}^1|q|^k\psi_N(q)\,dq&\leq2\kappa_N^{-k/2}+2\int_{\kappa_N^{-1/2}}^1q^k\psi_N(q)\,dq\nonumber\\
        &=2\kappa_N^{-k/2}+\kappa_N^{-k/2}\int_{\kappa_N^{-1/2}\leq|q|<1}\psi_N(q)\,dq+k\int_{\kappa_N^{-1/2}}^1r^{k-1}\int_{r\leq|q|<1}\psi_N(q)\,dq\,dr\nonumber\\
        &\leq2\kappa_N^{-k/2}+C\kappa_N^{-k/2}e^{-1/2}+C_k\int_{\kappa_N^{-1/2}}^\infty r^{k-1}e^{-\kappa_Nr^2/2}\,dr\nonumber\\
        &\leq C_k\kappa_N^{-k/2}.                                         
    \end{align}
    
    Choose
    \begin{align*}
        r_N=L\sqrt{\frac{\log N}{\kappa_N}}=o(N^{-1/4}),\qquad\bar q_N=\min\{q\in\mathcal Q_N:q\ge r_N\},
    \end{align*}
    where $L>0$ is a sufficiently large fixed constant. We split the second-moment identity at the discrete cutoff $\bar q_N$: the lattice terms use $q\in\mathcal Q_N$, while the spherical terms are integrated over continuous overlap variables.
    Since $\bar q_N=o(N^{-1/4})$, the same argument as \eqref{one cell comparison detailed} for $b>0$ still holds. 
    For $q\in\mathcal Q_N$ with $|t-q|\leq N^{-1}$ and $|q|\leq\bar q_N$, the corresponding cellwise error satisfies
    \begin{align*}
        N^{-1}+|q|+q^2+Nq^4
        \leq C\left(N^{-1}+|t|+t^2+Nt^4\right).
    \end{align*}
    Hence \eqref{continuous identity} and \eqref{negative-moments} give
    \begin{align}\label{negative center quadrature bound}
        \left|\sum_{\substack{q\in\mathcal Q_N\\|q|\le\bar q_N}}p_N(q)J_\beta(q)-\int_{|t|\le\bar q_N+N^{-1}}\psi_N(t)\,dt\right|
        &\leq C\int_{-1}^1\left(N^{-1}+|t|+t^2+Nt^4\right)\psi_N(t)\,dt\notag\\
        &\leq C\left(N^{-1}+\kappa_N^{-1/2}+\kappa_N^{-1}+N\kappa_N^{-2}\right)\\
        &\leq C_bN^{-1/3}(\log N)^{-1/4}.\notag
    \end{align}

    By \eqref{negative-tail},
    \begin{align}\label{negative spherical tail}
        \int_{\bar q_N+N^{-1}\le|q|<1}\psi_N(q)\,dq \le Ce^{-\kappa_Nr_N^2/2}=CN^{-L^2/2}. 
    \end{align}

    For $q\in\mathcal Q_N$, the method-of-types bound
    \begin{align*}
        p_N(q)\le e^{-NI(q)},\qquad I(q)=\frac12\big[(1+q)\log(1+q)+(1-q)\log(1-q)\big],
    \end{align*}
    together with Lemma~\ref{K beta derivative bound} and $I(q)\ge q^2/2+q^4/12$, gives
    \begin{align*}
        p_N(q)J_\beta(q)\le\exp\left\{-\frac N2(1-\beta^2)q^2-\frac N{12}q^4\right\}.
    \end{align*}

    If $q\in\mathcal Q_N$ and $|q|\ge\bar q_N\ge r_N$, then
    \begin{align*}
        N(1-\beta^2)q^2=(\kappa_N+3)q^2\ge L^2\log N.
    \end{align*}

    Since $|\mathcal Q_N|=N+1$, choosing $L>1$ sufficient large yields
    \begin{align}\label{negative cube tail}
        \sum_{\substack{q\in\mathcal Q_N\\|q|\ge\bar q_N}}p_N(q)J_\beta(q)\le N^{-1/2}.
    \end{align}

    Finally, apply the second-moment identity \eqref{second moment quadrature identity}. 
    Its center is controlled by \eqref{negative center quadrature bound}, and the absolute values of its spherical and lattice tail contributions are bounded by \eqref{negative spherical tail} and \eqref{negative cube tail}, respectively.
    Therefore
    \begin{align*}
        \mathbb E[(X_{N,\beta}-1)^2]\le C_bN^{-1/3}(\log N)^{-1/4}.
    \end{align*}
\end{proof}

\backmatter

\textbf{Funding.}
Liu S.H. was partially supported by the Fundamental Research Funds for the Central Universities DUT25RC(3)133.
Shao Q.M. was partially supported by National Nature Science Foundation of China NSFC 12031005 and Shenzhen Outstanding Talents Training Fund, China.

\bibliography{sk_references}

\end{document}